\documentclass[psamsfonts,reqno]{amsart}
\usepackage{amssymb,xy,eucal}
\usepackage[usenames]{color}

\usepackage{amscd}

\newcommand{\module}[1]{\lvert{#1}\rvert}

\newcommand{\ceil}[1]{\lceil{#1}\rceil}
\newcommand{\Ceil}[1]{\left\lceil{#1}\right\rceil}

\newcommand{\floor}[1]{\lfloor#1\rfloor}
\newcommand{\Floor}[1]{\left\lfloor#1\right\rfloor}

\usepackage{graphics}

\usepackage{mathrsfs}

\xyoption{all}

\theoremstyle{plain}

\newtheorem{lemma}{Lemma}[section]
\newtheorem{theorem}[lemma]{Theorem}

\newtheorem{corollary}[lemma]{Corollary}

\newtheorem{claim}{Claim}

\newtheorem*{theorem*}{Theorem}

\newtheorem*{stat}{\name}
\newcommand{\name}{testing}

\theoremstyle{definition}

\theoremstyle{remark}
\newtheorem{remark}[lemma]{Remark}
\newtheorem{notation}[lemma]{Notation}

\newtheorem*{remark*}{Remark}

\newenvironment{all}[1]{\renewcommand{\name}{#1}\begin{stat}}
                         {\end{stat}}

\newcommand{\qedc}{{\qed}~{\rm Claim~{\theclaim}.}}
\newcommand{\qedsc}{{\qed}~{\rm Claim.}}

\newenvironment{cproof}
{\begin{proof}[Proof of Claim.]}
{\qedc\renewcommand{\qed}{}\end{proof}}

\numberwithin{equation}{section}

\newcommand{\set}[1]{\{#1\}}
\newcommand{\setm}[2]{\set{#1\mid#2}}
\newcommand{\Set}[1]{\left\{#1\right\}}
\newcommand{\Setm}[2]{\Set{#1\mid#2}}

\newcommand{\famm}[2]{(#1\mid#2)}
\newcommand{\Famm}[2]{\left(#1\mid#2\right)}

\newcommand{\radi}[2]{r\left(#1,#2\right)}

\newcommand{\cO}{\mathcal{O}}

\newcommand{\C}{\mathbb{C}}
\newcommand{\N}{\mathbb{N}}
\newcommand{\Q}{\mathbb{Q}}
\newcommand{\R}{\mathbb{R}}
\newcommand{\Z}{\mathbb{Z}}

\newcommand{\res}{\mathbin{\restriction}}

\newcommand{\ignorer}[1]{}

\DeclareMathOperator{\Tr}{Tr}
\DeclareMathOperator{\dc}{\delta}

\newcommand{\modulo}[2]{R_{#2}(#1)}

\newcommand{\notdiv}{\nmid}
\newcommand{\qtr}{\Q^{{\rm tr}}}
\newcommand{\ztr}{\Z^{{\rm tr}}}

\newcommand{\Legendre}[2]{\left(\frac{{#1}}{{#2}}\right)}

\DeclareMathOperator{\JR}{JR}



\begin{document}

\title{Julia Robinson's numbers}

\author[P.~Gillibert]{Pierre Gillibert}

\address[P. Gillibert]{Institut f\"ur Diskrete Mathematik \& Geometrie, Technische Universit\"at Wien, Austria}
\email[P. Gillibert]{pgillibert@yahoo.fr}
\urladdr[P. Gillibert]{http://www.gillibert.fr/pierre/}

\author[G. Ranieri]{Gabriele Ranieri}
\address[G. Ranieri]{Number theory and algebra group, Valpara\'iso}
\email{gabriele.ranieri@pucv.cl}

\date{\today}

\begin{abstract}
We partially answer to a question of Vidaux and Videla by constructing an infinite family of rings of algebraic integers of totally real subfields of $\overline{\Q}$ whose Julia Robinson's Number is distinct from $4$ and $+ \infty$. Moreover the set of the Julia Robinson's Number that we construct is unbounded.  
\end{abstract}

\maketitle

\section{Introduction}

Let $R$ be the ring of the algebraic integers of a subfield of the field $\qtr$ of all totally real algebraic numbers. 
For every $t \in \R \cup \set{ + \infty}$, denote by $R_t$ the set of $\alpha \in R$ such that, all conjugates $\beta$ of $\alpha$ satisfy $0 < \beta < t$.
Set $A ( R ) = \setm{ t \in \R \cup \set{ + \infty }}{R_t \text{ is infinite}}$. Observe that $ A( R ) $ is either an interval or $\set{ + \infty }$.
The \emph{Julia Robinson Number} of $ R $ is defined as $\JR ( R )=\inf ( A ( R ) )$. Moreover $ R $ has the \emph{Julia Robinson Property} (JRP) if $\JR ( R ) \in A ( R )$.

Robinson, in \cite{Robinson}, generalizing her methods of \cite{Robinson2}, proves that if a ring $ R $ has the JRP, then the semi-ring $( \N ; 0, 1, +, .)$ is
first-order definable in $ R $ (hence the arithmetic of $R$ is as complicated as that of $\N$, and in particular $ R $ has an undecidable first order theory). There is a lot of examples of rings with the JRP.

If $ K $ is a totally real algebraic number field, then $\JR ( \cO_K ) = + \infty$ (see \cite{Robinson}). In particular, the first order theory of $\cO_K$ is undecidable. Note that Denef and Lipshitz prove a stronger result in \cite{Denef1,Denef2, DenefLipshitz}: if $ K $ is a totally real number field, or an extension of degree 2 of a totally real number field then $\N$ is Diophantine in $\cO_K$. It follows from the Matiyasevich's solution to Hilbert 10th problem \cite{Matiyasevich70}, that the existential positive first order theory of $\cO_K$ is undecidable. The result is even further extended by Shlapentokh in a serie of papers \cite{Shlapentokh94, Shlapentokh97, Shlapentokh00, Shlapentokh02, Shlapentokh03, Shlapentokh07} to large subrings of $K$ (see also Poonen \cite{Poonen03}, Poonen and Shlapentokh \cite{PoonenShlapentokh}). Mazur and Rubin prove in \cite{MazurRubin} a similar result for the algebraic integer rings of an arbitrary number fields, assuming that the Tate-Shafarevich conjecture holds. See Shlapentokh \cite{Shlapentokh11, Shlapentokh12, Shlapentokh15} for more results. See also Matiyasevich \cite{Matiyasevich11} for more on the Hilbert's 10th problem, and Moret-Bailly \cite{MoretBailly1, MoretBailly2} for similar problems in function fields.

On the other hands for infinite algebraic extension of $\Q$ far less is known. The Julia Robinson Property seems particularly useful. Vidaux and Videla observe in \cite{VidauxVidela2} that if a totally real algebraic field $ K $ has the Northcott property, then $\JR(\cO_K)=+\infty$. Using the result of Bombieri and Zannier in \cite{BombieriZannier}, Vidaux and Videla deduce that for every positive integer $ d $, the field generated by all totally real abelian extensions of $\Q$ of degree $ d $ has its ring of algebraic integers with $\JR$ equals to $+ \infty$. In particular the ring has undecidable first order theory. 

Denote by $\ztr$ the ring of all totally real algebraic integers. As noted by Robinson in \cite{Robinson}, it follows from Schur \cite{Schur} or Kronecker \cite{Kronecker} that $\JR(\ztr)=4$, and that $\ztr$ has the JRP (also see Robinson \cite{RRobinson} or Jarden and Videla \cite{JardenVidela}). In particular the first order theory  of $\ztr$ is undecidable (see \cite{Robinson}).
In \cite[Introduction]{VidauxVidela1}, Vidaux and Videla, starting from a remark of Robinson, ask the following question
\bigskip

\noindent {{\bf Question}: Can any algebraic number (or even real number) be realized as the JR number of a ring of integers in $\qtr$?}
\bigskip

Then, they find a family of rings $\cO$ contained in $\ztr$, such that $\JR ( \cO )$ is distinct from $4$ and $+\infty$ (see \cite[Theorem 1.4]{VidauxVidela1} for the details). On the other hand, they do not know if at least an element of the family is the ring of the algebraic integers of its quotient field. 
The paper \cite{VidauxVidela1} is also interesting because it contains a very precise bibliography on undecidability problem over algebraic rings. Finally, very recently, Vidaux and Videla (see \cite{VidauxVidela2}) discovered an interesting relation between the finiteness of the Julia Robinson Number and the Northcott Property. Starting from this, Widmer defined the Northcott Number and gave interesting examples of families of algebraic rings with bounded Northcott Numbers (see \cite{Widmer} for the details).    

In this paper we find infinitely many real numbers $t > 4$ such that $ t $ is the Julia Robinson Number of a ring of the algebraic integers of a subfield of $\qtr$. This yields the first example of ring of the algebraic integers of a totally real field with $\JR$ distinct from $4$ and $+\infty$. More precisely, we prove the following result.

\begin{all}{Theorem \ref{T:exemplenombreJR}}
Let $t\ge 1$ be a square-free odd number. Then the following statement holds.
\begin{enumerate}
\item There are infinitely many fields $ K $ such that $\cO_K$ has Julia Robinson's number $\Ceil{2\sqrt{2t}}+2\sqrt{2t}$.
\item There are infinitely many fields $ K $ such that $\cO_K$ has Julia Robinson's number $8t$.
\end{enumerate}
\end{all}

To prove Theorem~\ref{T:exemplenombreJR}, we choose $0 < a < b$ coprime integers satisfying certain conditions (see for instance Section \ref{S:JR}) and we set $\alpha = \frac{ a + i\sqrt{b^2 - a^2}}{b}$. Then $\alpha$ is a quadratic number with absolute value equal to $1$. We recall that a number field $L$ is $CM$ if it is a quadratic imaginary extension of a totally real field (or equivalently $ L $ is not contained in $\R$ and there exists $\iota$ in ${\rm Aut} ( L/ \Q )$ such that for every embedding of $ L $ over $\C$, $\iota$ is the complex conjugation, see for instance \cite[Chapter 4]{Washington}). Amoroso and Nuccio \cite[Proposition 2.3]{AmorosoNuccio} prove that a field is $CM$ if and only if it is generated over $\Q$ by an element whose all conjugates have absolute value equal to $1$. Then for every positive integer $n$, the field generated by an $n$th root of $\alpha$ is a $CM$ field, and so it contains a totally real field of index $2$. We choose a positive integer $r$ such that, for every prime number dividing $b$, $v_p ( b )$ is coprime with $r$ and, for every positive integer $k$, we set $K_{r^k}$ the maximal totally real field contained in the field generated by a $r^k$-th root of $\alpha$. Let $K$ be the union of all $K_{r^k}$ and $\cO_K$ its ring of the algebraic integers. By characterizing the algebraic integers of $K_{r^k}$ for every $k$, we get a very precise description of the algebraic integers of $ K $ and we compute the JR Number of $K$ (see Theorem \ref{T:MainJuliaRobinsonNumber}). Moreover $\cO_K$ has the JRP and so the first order theory of $\cO_K$ is undecidable (see Corollary~\ref{C:indec}). 

In Section~\ref{sec7} we use some results of analytic number theory to get the explicit examples of Julia Robinson's number given in Theorem~\ref{T:exemplenombreJR}.

\section{Notations}\label{sec1}
Let $F$ be an algebraic extension of $\Q$, we denote by $\cO_F$ the ring of all algebraic integers of $F$. Assume that $F$ is a number field. Let $p$ be a prime number, $\beta$ be in $F^\ast$, and $P$ be a prime ideal of $\cO_F$ over $p$, let $e$ be its ramification index over $\Q$ and $\lambda \in \Z$ be the exponent of $P$ in the factorization of the fractional ideal $( \beta )$.  We define the $P$-adic valuation of $\beta$ as $\frac{\lambda}{e}$. Then in particular, the $P$-adic valuation of $\beta$ should not be an integer.
We always denote by $v_p$ the $p$-adic valuation and we always say that the $P$-adic valuation extends the $p$-adic valuation. 

For every real number $\gamma$, we set $\ceil{\gamma}$ the smallest integer $\geq \gamma$, and $\floor{\gamma}$ the largest integer $\le\gamma$.

\section{Some preliminaries}\label{sec2}

Given an integer $b\ge 1$, and given $p_1^{n_1}p_2^{n_2}\dots p_s^{n_s}$ the prime decomposition of $ b $, and $u\ge 0$ in $\mathbb{Q}$, we denote
\[
\radi{b}{u}=p_1^{\ceil{n_1u}}p_2^{\ceil{n_2u}}\dots p_s^{\ceil{n_su}}\,.
\]
The next Lemma gives some basic properties of $\radi{b}{u}$.

\begin{lemma}\label{L:radi}
Let $b\ge 1$ be an integer, let $u\ge 0$ be in $\mathbb{Q}$. Then $\radi{b}{u}$ is an integer, and $\radi{b}{u}b^{-u}$ is an algebraic integer. Also note that $\radi{b}{0}=1$ and $\radi{b}{1}=b$.
\end{lemma}

Let $ K $ be a number field and let $\alpha$ be an element of $ K $. Choose an $n$th root of $\alpha$ and denote it by $\alpha^{\frac{1}{n}}$.  We need to study some properties of a linear combination with coefficients in $K$ of some powers of   $\alpha^{\frac{1}{n}}$. We do this in the following Lemmas.

\begin{lemma}\label{L:preval}
Let $n\ge 1$ be an integer, let $ t $ be an integer relatively prime to $n$, $-n+1\le k_1,k_2\le n-1$, and $\ell_1,\ell_2$ in $\mathbb{Z}$. Assume that $\ell_1+\frac{k_1t}{n}=\ell_2+\frac{k_2t}{n}$. Then one of the following holds:
\begin{enumerate}
\item $k_1=k_2$.
\item $k_1<0<k_2$ and $k_2=k_1+n$.
\item $k_2<0<k_1$ and $k_1=k_2+n$.
\end{enumerate}
\end{lemma}

\begin{proof}
Assume that $\ell_1+\frac{k_1t}{n}=\ell_2+\frac{k_2t}{n}$. Hence $n\ell_1+k_1t=n\ell_2+k_2t$. Thus $(k_1-k_2)t$ is a multiple of $n$. Then it follows, as $t$ is relatively prime to $n$, that $k_1-k_2$ is a multiple of $n$. Since $-2n+2\le k_1-k_2\le 2n-2$, we have that $k_1-k_2\in\set{-n,0,n}$.

If $k_1-k_2=0$ then $k_1=k_2$.

If $k_1-k_2=n$, then $k_1=n+k_2$. Moreover, we have $k_1=n+k_2\ge n -n+1>0$ and $k_2=k_1-n\le n-1-n<0$.

If $k_1-k_2=-n$, then $k_2=n+k_1$ and $k_1=-n+k_2\le -n +n-1<0$. Similarly $k_2=k_1+n\ge n-1-n<0$.
\end{proof}

\begin{lemma}\label{L:groslambda}
Let $ K $ be a number field, $n$ be a natural number, $ p $ be a prime number, $\alpha$ be in $K$. Let $\nu_{1, p}$, and $\nu_{2, p}$ be valuations over $ K $ extending $ v_p $ to $ K $, let $\famm{\lambda_k}{-n+1\le k\le n-1}$ be a family in $ K $. Assume that the following conditions hold:
\begin{enumerate}
\item $\nu_{1, p}(\alpha)$ is a positive integer relatively prime to $n$.
\item $\nu_{2, p}(\alpha)$ is a negative integer relatively prime to $n$.
\item 
\[\gamma = \sum_{k=-n+1}^{n-1}\lambda_k\alpha^{\frac{k}{n}}\quad\text{is an algebraic integer.}
\]
\item For all $-n+1\le k\le n-1$, we have $\nu_{1, p}( \lambda_k )=\nu_{2, p}(\lambda_k)$ is an integer.
\end{enumerate}
Then for all $-n+1\le k\le n-1$, we have $\nu_{1, p}(\lambda_k)\ge\Ceil{\frac{-k\nu_{2, p}(\alpha)}{n}}$, $\nu_{1, p}(\lambda_{k})\ge\Ceil{\frac{-k\nu_{1, p}(\alpha)}{n}}$, and $\nu_{1, p}(\lambda_0)\ge 0$.
\end{lemma}

\begin{proof}
By conditions (1) and (2), $K ( \alpha^{\frac{1}{n}} )/ K$ is totally ramified over $\nu_{1, p}$ and $\nu_{2, p}$. Then, $\nu_{1, p}$ and $\nu_{2, p}$ extend in a unique way to $K(\alpha^{\frac{1}{n}})$. Denote $I=\setm{k\in\mathbb{Z}}{-n+1\le k\le n-1}$. 
As $ \gamma $ is an algebraic integer, it follows that $\nu_{1, p}(\gamma)\ge 0$, and $\nu_{2, p}(\gamma)\ge 0$. Given $k\in I$, the following equalities hold: 
\[
\nu_{1, p}(\lambda_k\alpha^{\frac{k}{n}})=\nu_{1, p}(\lambda_k)+\frac{k\nu_{1, p}(\alpha)}{n}\,.
\]
\[
\nu_{2, p}(\lambda_k\alpha^{\frac{k}{n}})=\nu_{2, p}(\lambda_k)+\frac{k\nu_{2, p}(\alpha)}{n}=\nu_{1, p}(\lambda_k)+\frac{k\nu_{2, p}(\alpha)}{n}\,.
\]

Denote
\[
S_1=\Setm{\nu_{1, p}(\lambda_k)+\frac{k\nu_{1, p}(\alpha)}{n}}{k\in I}
=\Setm{\nu_{1, p}(\lambda_k\alpha^{\frac{k}{n}})}{k\in I}
\]
\[
S_2=\Setm{\nu_{1, p}(\lambda_k)+\frac{k\nu_{2,p}(\alpha)}{n}}{k\in I}
=\Setm{\nu_{2, p}(\lambda_k\alpha^{\frac{k}{n}})}{k\in I}.
\]
Let $u$ be the smallest element of $S_1\cup S_2$. Assume that $u<0$ and that $u\in S_1$. Fix $k_1$ such that $u=\nu_{1, p}(\lambda_{k_1})+\frac{k_1\nu_{1, p}(\alpha)}{n}$.

Assume that for all $k\not=k_1$, we have $u<\nu_{1, p}(\lambda_k\alpha^{\frac{k}{n}})$, then $\nu_{1, p}(\gamma)=u<0$; a contradiction. Therefore, there is $k_2\in I$ such that $k_2\not=k_1$ and the following equalities hold:
\[
\nu_{1, p}(\lambda_{k_1})+\frac{k_1\nu_{1, p}(\alpha)}{n}=u=\nu_{2, p}(\lambda_{k_2}\alpha^{\frac{k_2}{n}})=\nu_{1, p}(\lambda_{k_2})+\frac{k_2\nu_{1, p}(\alpha)}{n}\,.
\]
Up to an exchange of $k_1$ and $k_2$, we can assume that $k_1<k_2$, hence it follows from Lemma~\ref{L:preval} that $k_1<0<k_2$ and $k_2=k_1+n$. As $k_2>0$ and $\nu_{1, p}(\alpha)>0>\nu_{2, p}(\alpha)$, it follows that
\[
\nu_{1, p}(\lambda_{k_2})+\frac{k_2\nu_{2, p}(\alpha)}{n}< \nu_{1, p}(\lambda_{k_2}) < \nu_{1, p}(\lambda_{k_2})+\frac{k_2\nu_{1, p}(\alpha)}{n}=u\,.
\]
Hence $u$ is not the smallest element of $S_1\cup S_2$; a contradiction. With a similar argument if $u\in S_2$, we can find another contradiction. It follows that $u\ge 0$. 

For all $k\in I$, we have $\nu_{1, p}(\lambda_k)+\frac{k\nu_{1,p}(\alpha)}{n}\ge u\ge 0$. Hence $\nu_{1, p}(\lambda_k)\ge -\frac{k\nu_{1,p}(\alpha)}{n}$, therefore, as $\nu_{1, p}(\lambda_k)$ is an integer it follows that $\nu_{1, p}(\lambda_k)\ge \Ceil{\frac{-k\nu_{1,p}(\alpha)}{n}}$. A similar argument yields $\nu_{1, p}(\lambda_k)=\nu_{2, p}(\lambda_k)\ge \Ceil{\frac{-k\nu_{2,p}(\alpha)}{n}}$.
\end{proof}

\begin{lemma}\label{L:dim}
Let $K$ be a number field, $n\in\mathbb{N}$, $p$ be a prime number, $\alpha\in K$. Let $\nu_p$ be a valuation over $K$ extending $v_p$ to $K$. Assume that the following conditions hold:
\begin{enumerate}
\item $\nu_p(\alpha)$ is an integer relatively prime to $n$.
\item For all $\lambda\in K^*$, $\nu_p(\lambda)$ is an integer.
\end{enumerate}
Then $[K(\alpha^{\frac{1}{n}}):K]=n$.
\end{lemma}

\begin{proof}
We extend $\nu_p$ to $K(\alpha^{\frac{1}{n}})$. By abuse of notation, we still denote $\nu_p$ the extension. Let $\famm{\lambda_k}{0\le k\le n-1}$ be a family of $K$ such that: $\sum_{k=0}^n\lambda_k\alpha^{\frac{k}{n}}=0$.

Note that, for all $0\le k\le n$, we have $\nu_p(\lambda_k\alpha^{\frac{k}{n}})=\nu_{1,p}(\lambda_k)+\frac{k}{n}\nu_p(\alpha)$.

Set $I=\setm{k\in\mathbb{N}}{k\le n-1\text{ and }\lambda_k\not=0}$. It follows from Lemma~\ref{L:preval} that for all $k\not=\ell$ in $I$ we have $\nu_p(\lambda_k)+\frac{k}{n}\nu_p(\alpha)\not=\nu_p(\lambda_\ell)+\frac{\ell}{n}\nu_p(\alpha)$. Therefore, if $I$ is not empty, $\nu_p(\sum_{k=0}^n\lambda_k\alpha^{\frac{k}{n}})$ is the minimum of $\nu_p(\lambda_k)+\frac{k}{n}\nu_p(\alpha)$ for $k\in I$. Therefore $\nu_p(0)\not=\infty$; a contradiction.
\end{proof}

\begin{lemma}\label{L:trace}
Let $K$ be a field, let $\alpha$ be in $K$, let $n$ be an integer $\geq 2$. Assume that $[K(\alpha^{\frac{1}{n}}):K] = n$. Let $k\in\mathbb{Z}$ be such that $k$ is not a multiple of $n$. Then $\Tr_{K(\alpha^{\frac{1}{n}})/K}(\alpha^{\frac{k}{n}})=0$.
\end{lemma}

\begin{proof}
Set $d = ( k, n )$ and let $u,v$ in $\mathbb{Z}$ be such that $ku+nv=d$. Since $[K ( \alpha^{\frac{1}{n}} ) : K] = n$, we have that $X^{\frac{n}{d}}-\alpha^{\frac{k}{d}}$ is the minimal polynomial of $\alpha^{\frac{k}{d}}$ over $K$.
It follows that $\Tr_{K( \alpha^{\frac{1}{n}} )/K}(\alpha^{\frac{k}{n}})=0$.
\end{proof}

In the sequel, we are  interested in the case when $K$ is a quadratic imaginary field and $\alpha \in K$ such that $\alpha$ has absolute value equal to $1$. In the following Lemma we apply some of the previous results in this particular case.

\begin{lemma}\label{L:entieralgebrique}
Let $a,b$ be relatively prime in $\mathbb{N}$, with $a<b$. Let $\alpha$ be equal to $\frac{a+i\sqrt{b^2-a^2}}{b}$. Let $n\ge 1$ be an integer. Assume that $\alpha^{\frac{1}{n}}$ generates an extension of degree $n$ over $\mathbb{Q}(i\sqrt{b^2-a^2})$. Let $\famm{\lambda_k}{-n+1\le k\le n-1}$ be a family in $\mathbb{Q}$. Assume that
\[\sum_{k=-n+1}^{n-1}\lambda_k\alpha^{\frac{k}{n}}\quad\text{is an algebraic integer.}
\]
Then, for all $1\le k\le n-1$, both $nb\lambda_{-k}+n\lambda_{n-k}(a+i\sqrt{b^2-a^2})$ and $nb\lambda_{k}+n\lambda_{k-n}(a-i\sqrt{b^2-a^2})$ are algebraic integers, and $n\lambda_0$ is an integer.
\end{lemma}

\begin{proof}
Set $\gamma=\sum_{k=-n+1}^{n-1}\lambda_k\alpha^{\frac{k}{n}}$. Set $K=\mathbb{Q}(i\sqrt{b^2-a^2})$ and set $L=K(\alpha^{\frac{1}{n}})$. It follows from Lemma~\ref{L:trace} that
\begin{equation}\label{E:tracezero}
\Tr_{L/K}(\alpha^{\frac{k}{n}})=0\,,\quad\text{for all $k\in\mathbb{Z}$ which is not a multiple of $n$}.
\end{equation}
Therefore $\Tr_{L/K}(\gamma)=\Tr_{L/K}(\lambda_0)=n\lambda_0$. However, $\gamma$ is an algebraic integer, therefore $n\lambda_0$ is an algebraic integer. As $\lambda_0\in\mathbb{Q}$, it follows that $n\lambda_0$ is an integer.

Let $1\le k\le n-1$ be an integer. The following equalities hold:
\[
(b\alpha^{\frac{k}{n}})^n=b^n\alpha^k=b^n\left(\frac{a+i\sqrt{b^2-a^2}}{b}\right)^k=b^{n-k}(a+i\sqrt{b^2-a^2})^k\,.
\]
Therefore $b\alpha^{\frac{k}{n}}$ is an algebraic integer. The following equalities hold
\begin{align*}
\Tr_{L/K}(b\alpha^{\frac{k}{n}}\gamma)&=\Tr_{L/K}\left(\sum_{\ell=-n+1}^{n-1}b\lambda_\ell\alpha^{\frac{\ell+k}{n}}\right)\\
&=\Tr_{L/K}(b\lambda_{-k})+\Tr_{L/K}(b\lambda_{n-k}\alpha)\,,&&\text{by \eqref{E:tracezero}.}\\
&=nb\lambda_{-k}+nb\lambda_{n-k}\alpha\,,&&\text{as $b\lambda_{-k}$ and $b\lambda_{n-k}\alpha$ are in $K$.} \\
&=nb\lambda_{-k}+n\lambda_{n-k}(a+i\sqrt{b^2-a^2}).
\end{align*}
However, $b\alpha^{\frac{k}{n}}\gamma$ is an algebraic integer, therefore $nb\lambda_{-k}+n\lambda_{n-k}(a+i\sqrt{b^2-a^2})$ is an algebraic integer.

Similarly,
\[
(b\alpha^{\frac{-k}{n}})^n=b^n\alpha^{-k}=b^n\left(\frac{a+i\sqrt{b^2-a^2}}{b}\right)^{-k}=
b^n\left(\frac{a-i\sqrt{b^2-a^2}}{b}\right)^{k}=b^{n-k}(a-i\sqrt{b^2-a^2})^k\,.
\]
Hence $b\alpha^{\frac{-k}{n}}$ is an algebraic integer. The following equalities hold
\begin{align*}
\Tr_{L/K}(b\alpha^{\frac{-k}{n}}\gamma)&=\Tr_{L/K}\left(\sum_{\ell=-n+1}^{n-1}b\lambda_\ell\alpha^{\frac{\ell-k}{n}}\right)\\
&=\Tr_{L/K}(b\lambda_{k})+\Tr_{L/K}(b\lambda_{k-n}\alpha^{-1})\,,&&\text{by \eqref{E:tracezero}.}\\
&=nb\lambda_{k}+nb\lambda_{k-n}\alpha^{-1}\,,&&\text{as $b\lambda_{k}$ and $b\lambda_{k-n}\alpha^{-1}$ are in $K$.} \\
&=nb\lambda_{k}+n\lambda_{k-n}(a-i\sqrt{b^2-a^2}).
\end{align*}
However, $b\alpha^{\frac{-k}{n}}\gamma$ is an algebraic integer, therefore $nb\lambda_{k}+n\lambda_{k-n}(a-i\sqrt{b^2-a^2})$ is an algebraic integer.
\end{proof}

Let $K$ and $\alpha$ be as in the previous Lemma. Since $\alpha$ has absolute value equal to $1$, there exists $\theta \in \R$ such that $\alpha = \cos ( \theta ) + i \sin ( \theta )$. In various cases, we shall need some basic trigonometric formulas, as Euler's formula, from which it follows the following result.

\begin{lemma}\label{L:sommecos}
Let $\lambda$, $\mu$, $x$ and $y$ in $\mathbb{R}$. Then:
\[
\lambda\cos(x)+\mu\cos(y)=(\lambda+\mu)\cos\left(\frac{x+y}{2}\right)\cos\left(\frac{x-y}{2}\right)+(\mu-\lambda)\sin\left(\frac{x+y}{2}\right)\sin\left(\frac{x-y}{2}\right)
\]
\end{lemma}

Similarly the following lemma is a consequence of Euler's formula.

\begin{lemma}\label{L:Euler2}
For all $\lambda,\mu,\theta,n\not=0$, and $k$ in $\R$ the following equality holds.
\begin{multline*}
\left(\lambda\cos\frac{k\theta}{n} + \mu\cos\frac{(n-k)\theta}{n}\right)^2 
= \frac{1}{2}\Bigg(\lambda^2+\mu^2+2\lambda\mu\cos\theta + \\
 + (\lambda^2-\mu^2) \cos\frac{2k\theta}{n} + 2(\lambda\mu+\mu^2\cos\theta) \cos\frac{(n-2k)\theta}{n}\Bigg)
\end{multline*}
\end{lemma}

\begin{lemma}\label{L:radiEntierAlg}
Let $0<a<b$ be relatively prime integers, set $\alpha=\frac{a+i\sqrt{b^2-a^2}}{b}$. Assume that $b$ is even. Let $q\ge 0$ be in $\mathbb{Q}$. Then $\radi{\frac{b}{2}}{q}\alpha^{q}$ is an algebraic integer.
\end{lemma}

\begin{proof}
As $b$ is even, it follows that $\frac{a+i\sqrt{b^2-a^2}}{2}$ is an algebraic integer. The following equalities hold
\begin{align*}
\radi{\frac{b}{2}}{q}\alpha^{q} &= \radi{\frac{b}{2}}{q} \left(\frac{a+i\sqrt{b^2-a^2}}{b}\right)^q\\
&=\radi{\frac{2}{b}}{q} \left(\frac{2}{b}\right)^{-q}\left(\frac{a+i\sqrt{b^2-a^2}}{2}\right)^q.
\end{align*}
From Lemma~\ref{L:radi} we have $\radi{\frac{2}{b}}{q} \left(\frac{2}{b}\right)^{-q}$ is an algebraic integer, moreover $\frac{a+i\sqrt{b^2-a^2}}{2}$ is an algebraic integer. Therefore $\radi{\frac{b}{2}}{q}\alpha^{q}$ is an algebraic integer.
\end{proof}

\begin{lemma}\label{L:cosqtheta}
Let $0<a<b$ be relatively prime integers, set $\theta=\arccos(\frac{a}{b})\in]0,\frac{\pi}{2}[$. Let $n\ge 2$ be an integer and $1\le k\le n-1$. Let $-\frac{1}{2}\le q\le \frac{1}{2}$ be in $\mathbb{Q}$. Then $2\sqrt{b}\cos(q\theta)$ and $2\sqrt{b}\sin(q\theta)$ are algebraic integers.
\end{lemma}

\begin{proof}
We can assume that $q=\frac{u}{v}$ with $u,v$ positive integers. The following equalities hold:
\[
(\sqrt{b}e^{i\frac{u}{v}\theta})^v=b^{\frac{v}{2}} e^{iu\theta}=b^{\frac{v}{2}-u}(be^{i\theta})^u=b^{\frac{v}{2}-u}\left(a+i\sqrt{b^2-a^2}\right)^u\,.
\]
As $\frac{v}{2}\ge u$ it follows that $b^{\frac{v}{2}-u}$ is an algebraic integer, moreover $(a+i\sqrt{b^2-a^2})^u$ is an algebraic integer. Therefore $\sqrt{b}e^{i\frac{u}{v}\theta}$ is an algebraic integer.

Similarly 
\[
(\sqrt{b}e^{-i\frac{u}{v}\theta})^v=b^{\frac{v}{2}} e^{-iu\theta}=b^{\frac{v}{2}-u}(be^{-i\theta})^u=b^{\frac{v}{2}-u}\left(a-i\sqrt{b^2-a^2}\right)^u\,.
\]
Hence $\sqrt{b}e^{-i\frac{u}{v}\theta}$ is algebraic integer. Therefore $2\sqrt{b}\cos(q\theta)=\sqrt{b}(e^{iq\theta}+e^{-iq\theta})$ is an algebraic integer. Similarly $2\sqrt{b}\sin(q\theta)$ is an algebraic integer.
\end{proof}

\begin{lemma}\label{L:SommeCosEntierAlgebrique}
Let $0<a<b$ be relatively prime integers, let $\theta=\arccos(\frac{a}{b})\in]0,\frac{\pi}{2}[$. Let $n\ge 2$ be an integer and $1\le k\le n-1$. Let $\lambda$ and $\mu$ in $\mathbb{Q}$ be such that $( \lambda+\mu )\sqrt{b+a}$ and $( \mu-\lambda )\sqrt{b+a}$ are algebraic integers. Then $4b( \lambda \cos(\frac{k\theta}{n}) + \mu \cos(\frac{(n-k)\theta}{n}))$ is an algebraic integer.
\end{lemma}

\begin{proof}
Set:
\[
\gamma = \lambda \cos\left(\frac{k\theta}{n}\right) + \mu \cos\left(\frac{(n-k)\theta}{n}\right)\,.
\]
It follows from Lemma~\ref{L:sommecos} that 
\begin{align*}
\gamma &= ( \lambda+\mu )\cos\left(\frac{\theta}{2}\right)\cos\left(\frac{(2k-n)\theta}{2n}\right) + ( \mu - \lambda )\sin\left(\frac{\theta}{2}\right)\sin\left(\frac{(2k-n)\theta}{2n}\right)\,.
\end{align*}
Note that, since $\theta=\arccos(\frac{a}{b})\in]0,\frac{\pi}{2}[$,
\[
\cos\left(\frac{\theta}{2}\right)=\sqrt{\frac{1}{2}(1+\cos(\theta))}=
\sqrt{\frac{1}{2}(1+\frac{a}{b})}=\sqrt{\frac{b+a}{2b}}\,.
\]
Similarly
\[
\sin\left(\frac{\theta}{2}\right)=\sqrt{\frac{b-a}{2b}}\,.
\]
Therefore
\begin{align*}
2\sqrt{2}b\gamma &= 
2\sqrt{2}b\left( ( \lambda+\mu )\sqrt{\frac{b+a}{2b}}\cos\left(\frac{(2k-n)\theta}{2n}\right) + ( \mu - \lambda )\sqrt{\frac{b-a}{2b}}\sin\left(\frac{(2k-n)\theta}{2n}\right)\right)\\
&= ( \lambda+\mu )\sqrt{b+a}2\sqrt{b}\cos\left(\frac{(2k-n)\theta}{2n}\right) + 
( \mu - \lambda )\sqrt{b-a}2\sqrt{b}\sin\left(\frac{(2k-n)\theta}{2n}\right).
\end{align*}
Note that $-\frac{1}{2}\le \frac{(2k-n)}{2n} \le \frac{1}{2}$. It follows, from Lemma~\ref{L:cosqtheta}, that $2\sqrt{b}\cos\left(\frac{(2k-n)\theta}{2n}\right)$ and $2\sqrt{b}\sin\left(\frac{(2k-n)\theta}{2n}\right)$ are algebraic integers. Moreover, we assumed that $( \lambda+\mu )\sqrt{b+a}$ and $( \mu - \lambda )\sqrt{b-a}$ are algebraic integers, therefore $4b\gamma$ is an algebraic integer.
\end{proof}

\begin{lemma}\label{L:existencevaluations}
Let $0<a<b$ be coprime integers. Set $\alpha = \frac{ a + i\sqrt{b^2 - a^2}}{b}$. Let $ p $ be a prime number dividing $b$ and, if $p = 2$ suppose that $4$ divides $b$.
Then $( p )$ splits in $\Q ( \alpha )$.
Let $P_1$ and $P_2$ be the prime ideals over $p$ and $\nu_{1, p}$, respectively $\nu_{2, p}$ the associated valuations.
Then, if $p$ is odd $\nu_{1, p} ( \alpha ) = - \nu_{2, p} ( \alpha ) = v_p ( b )$, where, as before, $v_p$ is the $p$-adic valuation.
On the other hand, if $p = 2$, $\nu_{1, 2}  ( \alpha ) = - \nu_{2, 2}  ( \alpha ) = v_2 ( b ) -1$.  
\end{lemma}

\begin{proof}
First observe that $\alpha \overline{\alpha} = 1$.
Let $R$ be the ring of the algebraic integers of $\Q ( \alpha )$.
Suppose first that $p$ is odd. 
Since $p$ does not divide $b^2 - a^2$, $( p )$ is not ramified in $ R $.
Suppose that $( p )$ divides $a + i\sqrt{b^2 - a^2}$ in $R$. 
Then, since $( p )$ is fixed by the complex conjugation, $p$ divides $b$ and $\alpha \overline{\alpha} = 1$, we should have $p$ divides $a + i\sqrt{b^2 - a^2}$ and $a - i\sqrt{b^2 - a^2}$ in $R$.
In this case, $( p )$ should divide $2a$ and this is not possible because $p$ is odd and $p$ does not divide $a$.
This proves that $p R  = P_1 P_2$, with $P_1$ and $P_2$ distinct prime ideals of $R$ and that, for every $i$, $j$ in $\{1, 2\}$, $i \neq j$, if $P_i$ divides $a + i\sqrt{b^2 - a^2}$, then $P_j$ divides $a - i\sqrt{b^2 - a^2}$.
Suppose that $P_1$ divides $a + i\sqrt{b^2 - a^2}$ and $P_2$ divides $a - i\sqrt{b^2 - a^2}$ and let $r$ be the exponent of $p$ in the prime factorization of $b$.
Then, since $\alpha \overline{\alpha} = 1$, by the unique prime decomposition of fractional ideal of $R$ with integer exponents, we get that the fractional ideal $( \alpha )$ generated by $\alpha$ can be written as
\begin{equation}\label{rel1}
( \alpha ) = P_1^{2 r } ( p )^{-r} I = P_1^r P_2^{-r} I,
\end{equation}
with $I$ fractional ideal coprime with $( p )$.
In fact if $t$ is the exponent of $P_1$ in the factorization of $a + i\sqrt{b^2 - a^2}$, since $P_2 = \overline{P_1}$ we have that $t$ is the exponent of $P_2$ in the factorization of $a - i\sqrt{b^2 - a^2}$.
Moreover, we have $( a + i\sqrt{b^2 - a^2} ) ( a - i\sqrt{b^2 - a^2} ) = b^2$.
Since $p^{2r}$ divides $b^2$ and $p^{2r + 1}$ does not divide $b^2$, we have $t = 2r$, proving (\ref{rel1}).

Let $\nu_{1, p}$ be respectively $\nu_{2, p}$ the valuations associated to $P_1$ respectively $P_2$.
Thus by (\ref{rel1}), $r = \nu_{1, p} ( \alpha ) = -\nu_{2, p} ( \alpha ) = v_p ( b )$.

Suppose now that $p = 2$.
Since $4$ divides $b$ and $a$ is coprime with $b$, $a^2 - b^2 \equiv 1 \mod ( 8 )$.
Then by the reciprocity quadratic law $( 2 )$ splits in the ring of the algebraic integers of $\Q ( \alpha )$.
Observe that $\frac{a + i\sqrt{b^2 - a^2}}{2}$ is a root of the polynomial $x^2 -a x + b^2$ with integer coefficients.
Then $( 2 )$ divides $a + i\sqrt{b^2 - a^2}$ and $a - i\sqrt{b^2 - a^2}$.
Moreover $( 4 )$ does not divide $a + i\sqrt{b^2 - a^2}$ and $a - i\sqrt{b^2 - a^2}$ because $\frac{a + i\sqrt{b^2 - a^2}}{4}$ has the trace equals to $\frac{ a }{2}$, which is not an integer.
Write $( 2 ) = P_1 P_2$ with $P_1$ and $P_2$ prime ideals of the ring of the algebraic integers of $\Q ( \alpha )$.
Since $\alpha \overline{\alpha} = 1$, $( 2 )$ divides $a + i\sqrt{b^2 - a^2}$ and $( 4 )$ does not divide $a + i\sqrt{b^2 - a^2}$, we get that there exists a positive integer $c$ and ideal $I_1$, $I_2$ coprime with $( 2 )$ such that
\begin{equation}\label{rel2}
( \alpha ) =  ( 2 ) P_1^{2 c} I_1 ( b )^{-1} = ( 2 ) P_1^{2 c} I_1 ( 2 )^{-c-1} I_2^{-1} = P_1^c I_1 P_2^{-c} I_2^{-1}.  
\end{equation}  
Let $\nu_{1, 2}$ be respectively $\nu_{2, 2}$ the valuations associated to $P_1$ respectively $P_2$.
Thus by (\ref{rel2}), $\nu_{1, 2} ( \alpha ) = c = - \nu_{2, 2} ( \alpha )$ and $v_2 ( b ) = c+1 = \nu_{1, 2} ( \alpha ) + 1$.  
\end{proof}
      
The following Corollary immediately follows from Lemmas~\ref{L:dim}, Lemma~\ref{L:trace}, and Lemma~\ref{L:existencevaluations}.

\begin{corollary}\label{C:TraceNul}
Let $0<a<b$ be coprime integers. Set $\alpha = \frac{ a + i\sqrt{b^2 - a^2}}{b}$. Let $p$ be a prime number dividing $b$. Let $n>0$ be an integer. If $p$ is odd we assume that $v_p(b)$ is coprime to $n$. If $p$ is even we assume that $4$ divides $b$, and that $v_2(b)-1$ is coprime to $n$. Then $[\Q(\alpha^{\frac{1}{n}}):\Q(\alpha)]=n$, and for all integers $k$ not multiple of $n$ we have $\Tr_{\Q(\alpha^{\frac{1}{n}})/\Q(\alpha)}(\alpha^{\frac{k}{n}})=0$.
\end{corollary}

\begin{lemma}\label{L:basecosinus}
With the hypothesis of Corollary \textup{\ref{C:TraceNul}}, let $\theta=\arccos(\frac{a}{b})\in]0,\frac{\pi}{2}[$. Then $1,\cos\frac{\theta}{n},\cos\frac{2\theta}{n},\dots,\cos\frac{(n-1)\theta}{n}$ is a $\Q$-basis of $\Q(\cos\frac{\theta}{n})$. 

Moreover, $\Tr_{\Q(\cos\frac{\theta}{n})/\Q}\left(\cos\frac{k\theta}{n}\right)=0$ for all $1\le k\le n-1$.
\end{lemma}

\begin{proof}
By Corollary \ref{C:TraceNul}, $[\Q( \alpha^{\frac{1}{n}} ):\Q] = 2n$. Observe that $\cos\frac{\theta}{n} = \frac{\alpha^{\frac{1}{n}} + \overline{\alpha^{\frac{1}{n}}}}{2}$. So $X^2 - 2 \cos\frac{\theta}{n} X + 1$ is the minimal polynomial of $\alpha^{\frac{1}{n}}$ over $\Q ( \cos\frac{\theta}{n} )$, and so $[\Q ( \cos\frac{\theta}{n} ): \Q] = n$, as illustrated by the following diagram.
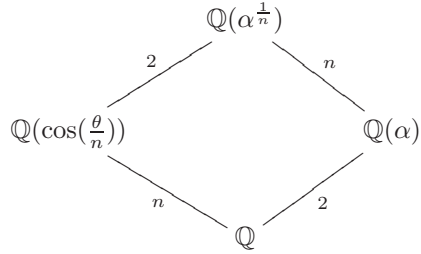
\begin{equation*}
\xymatrix{
 & \Q(\alpha^{\frac{1}{n}}) &\\
\Q(\cos(\frac{\theta}{n})) \ar@{-}[ur]^-{2} & & \Q(\alpha) \ar@{-}[ul]_-{n}\\
& \Q \ar@{-}[ur]_-{2} \ar@{-}[ul]^-{n}
}
\end{equation*}
To show that $1,\cos\frac{\theta}{n},\cos\frac{2\theta}{n},\dots,\cos\frac{(n-1)\theta}{n}$ is a $\Q$-basis of $\Q(\cos\frac{\theta}{n})$, it is sufficient to prove that they are linearly independent over $\Q$. Assume that we are given $\lambda_0,\dots,\lambda_{n-1}$ in $\Q$ such that $\sum_{k=0}^{n-1}\lambda_k\cos\frac{k\theta}{n} = 0$. The following equalities hold
\[
0=\sum_{k=0}^{n-1}\lambda_k\cos\frac{k\theta}{n}
=\lambda_0+ \sum_{k=1}^{n-1}\lambda_k\frac{\alpha^{\frac{k}{n}} + \overline{\alpha^{\frac{k}{n}}}}{2}
=\lambda_0+ \sum_{k=1}^{n-1}\lambda_k\frac{\alpha^{\frac{k}{n}} + \overline{\alpha}\alpha^{\frac{n-k}{n}}}{2}
\]
Therefore the following equality holds
\[
\lambda_0+ \sum_{k=1}^{n-1}\frac{\lambda_k + \lambda_{n-k}\overline\alpha}{2} \alpha^{\frac{k}{n}} = 0\,.
\]
Thus, as $1,\alpha^{\frac{1}{n}},\alpha^{\frac{2}{n}},\dots,\alpha^{\frac{n-1}{n}}$ is a basis of $\Q(\alpha^{\frac{1}{n}})$ over $\Q(\alpha)$, it follows that $\lambda_k=0$ for all $k$.

By Corollary \ref{C:TraceNul}, we have $\Tr_{\Q( \alpha^{\frac{k}{n}} )/ \Q}( \alpha^{\frac{k}{n}} ) = 0$.
Since $\cos\frac{k\theta}{n} = \frac{\alpha^{\frac{k}{n}} + \overline{\alpha^{\frac{k}{n}}}}{2}$, we immediately get $\Tr_{\Q(\cos\frac{\theta}{n})/\Q}\left( \cos\frac{k\theta}{n} \right) = 0$.
\end{proof}
 
\begin{lemma}\label{L:enumerationconjugue}
With the hypothesis of Corollary~\textup{\ref{C:TraceNul}}, Let $\theta=\arccos(\frac{a}{b})\in]0,\frac{\pi}{2}[$. Let $\famm{\lambda_k}{0\le k\le n-1}$ be a family in $\mathbb{Q}$. Then the conjugates of the element $\lambda_0+\sum_{k=1}^{n-1}\lambda_k\cos(\frac{k\theta}{n})$, counted with multiplicity in $\Q(\cos(\frac{\theta}{n}))$, are the elements $\lambda_0+\sum_{k=1}^{n-1}\lambda_k\cos(\frac{k(\theta+2\pi t)}{n})$ with $0\le t\le n-1$.
\end{lemma}

\begin{proof}
Note that $\alpha=e^{i\theta}=\frac{a + i\sqrt{b^2 - a^2}}{b}$, and $\overline\alpha=e^{-i\theta}=\alpha^{-1}$. We have $[\Q(\alpha^{\frac{1}{n}}):\Q(\alpha)]=n$, hence for each $0\le t\le n-1$ we have a unique morphism, denoted by $\varphi_t\colon \Q(\alpha^{\frac{1}{n}}) \to\C$ preserving $\Q(\alpha)$, such that $\varphi_t(\alpha^{\frac{1}{n}}) = e^{i\frac{2\pi t}{n}}\alpha^{\frac{1}{n}}$. The embedding $\Q(\alpha^{\frac{1}{n}}) \to\C$ preserving $\Q(\alpha)$, are $\varphi_0,\dots,\varphi_{n-1}$.
Since $\alpha$ is an algebraic number whose Galois conjugates have absolute value equal to $1$, then $\alpha^{\frac{1}{n}}$ has the same property. Then, by \cite[Proposition 2.3]{AmorosoNuccio}, $\Q( \alpha^{\frac{1}{n}} )$ is a $CM$ field. In particular the complex conjugation is well-defined and it is independent of any embedding of $\Q( \alpha^{\frac{1}{n}} )$ over $\C$. The embedding $\Q(\alpha^{\frac{1}{n}}) \to\C$ are $\varphi_0,\dots,\varphi_{n-1}$, and their complex conjugates $\overline\varphi_0,\dots,\overline\varphi_{n-1}$.
Let $k$ be in $\Z$, and $0\le t\le n-1$. The following equalities hold
\begin{equation}\label{E:enumconjugue}
\varphi_t\left(e^{\frac{ik\theta}{n}}\right) = \varphi_t\left(\alpha^{\frac{k}{n}}\right) = \varphi_t\left(\alpha^{\frac{1}{n}}\right)^k = (e^{i\frac{2\pi t}{n}}\alpha^{\frac{1}{n}})^k = e^{i\frac{2\pi kt}{n}}e^{\frac{ik\theta}{n}} = e^{i\frac{k(\theta+2\pi t)}{n}}.
\end{equation}
Therefore 
\begin{equation}\label{E:enumconjugue2}
\varphi_t(\cos(\frac{k\theta}{n})) = \frac{1}{2}\varphi_t\left(e^{\frac{ik\theta}{n}} + e^{\frac{-ik\theta}{n}}\right) = e^{i\frac{k(\theta+2\pi t)}{n}} + e^{-i\frac{k(\theta+2\pi t)}{n}} = \cos(\frac{k(\theta+2\pi t)}{n}).
\end{equation}

Set $K=\Q(\alpha^{\frac{1}{n}} + \overline{\alpha}^{\frac{1}{n}}) = \Q(\cos(\frac{\theta}{n}))$. Note that $\varphi_t\res K=\overline\varphi_t\res K$, for all $0\le t\le n-1$. Hence the morphisms $K\to\C$ are $\varphi_0\res K,\dots,\varphi_{n-1}\res K$. Set $\beta=\lambda_0+\sum_{k=1}^{n-1}\lambda_k\cos(\frac{k\theta}{n})$

Let $0\le t\le n-1$. The following equalities hold
\begin{align*}
\varphi_t\left(\beta\right) &= \varphi_t\left(\lambda_0+\sum_{k=1}^{n-1}\lambda_k\cos\left(\frac{k\theta}{n}\right)\right)\\
&= \lambda_0 + \sum_{k=1}^{n-1}\lambda_k\varphi_t\left(\cos\left(\frac{k\theta}{n}\right)\right)\\
&= \lambda_0 + \sum_{k=1}^{n-1}\lambda_k \cos\left(\frac{k(\theta+2\pi t)}{n}\right)\,. \tag*{\qed}
\end{align*}
\renewcommand{\qed}{}
\end{proof}

\begin{lemma}\label{L:caracterisionentieralgebrique}
Let $0<a<b$ be relatively prime integers, let $\theta=\arccos(\frac{a}{b})\in]0,\pi[$. Let $n\ge 2$. Let $\famm{\lambda_k}{0\le k\le n-1}$ be a family in $\mathbb{Q}$. Let $u,v$ in $\mathbb{N}^*$ be such that $\frac{b+a}{u^2}$ and $\frac{b-a}{v^2}$ are square-free integers.
Assume the following statements
\begin{itemize}
\item $4$ divides $b$.
\item $v_2 ( b )-1$ is relatively prime to $n$.
\item If $p$ is an odd prime number such that $p$ divides $n$ then $p$ divides $b$.
\item If $p$ is an odd prime number dividing $b$ then $v_p ( b )$ is relatively prime to $n$.
\end{itemize}
Then $\lambda_0+\sum_{k=1}^{n-1}\lambda_k\cos(\frac{k\theta}{n})$ is an algebraic integer if and only if for all $1\le k\le n-1$ the following statements hold 
\begin{enumerate}
\item $\lambda_0$ is an integer.
\item For all odd prime number $p$ such that $p$ divides $b$ we have 
$v_p(\lambda_k)\ge\Ceil{\frac{kv_p(b)}{n}}$.
\item $\nu_2(\lambda_k)\ge 1+\Ceil{\frac{k(v_2(b)-1)}{n}}$.
\item $u(\lambda_k+\lambda_{n-k})$ and $v(\lambda_k-\lambda_{n-k})$ are integers.
\end{enumerate}
\end{lemma}

\begin{proof}
As $b$ is even and $a$ is relatively prime to $b$, it follows that $b-a$ and $b+a$ are odd, thus $u$ and $v$ are odd. Set $\alpha=e^{i\theta}=\frac{a+i\sqrt{b^2-a^2}}{b}$.

Assume that $(1)-(4)$ hold. Let $1\le k\le n-1$. It follows from $(4)$ that $2uv\lambda_k$ is an integer. However by $(3)$ we have $v_2(\lambda_k)\ge 1$, therefore $uv\frac{\lambda_k}{2}$ is an integer.

It follows from $(2)$ and $(3)$ that $\radi{\frac{b}{2}}{\frac{k}{n}}$ divides $uv\frac{\lambda_k}{2}$, and so by Lemma~\ref{L:radiEntierAlg} we have $uv\frac{\lambda_k}{2}\alpha^{\frac{k}{n}}=uv\frac{\lambda_k}{2}e^{\frac{ki\theta}{n}}$ is an algebraic integer. Therefore $uv\lambda_k\cos\left(\frac{k\theta}{n}\right)=uv\frac{\lambda_k}{2}\left(e^{\frac{ki\theta}{n}}+e^{\frac{-ki\theta}{n}}\right)$ is an algebraic integer. Similarly $uv\lambda_{n-k}\cos\left(\frac{(n-k)\theta}{n}\right)$ is an algebraic integer. Therefore $uv\left(\lambda_{k}\cos\left(\frac{k\theta}{n}\right)+\lambda_{n-k}\cos\left(\frac{(n-k)\theta}{n}\right)\right)$ is an algebraic integer.

Lemma~\ref{L:SommeCosEntierAlgebrique} implies that $2b\left(\lambda_{k}\cos\left(\frac{k\theta}{n}\right)+\lambda_{n-k}\cos\left(\frac{(n-k)\theta}{n}\right)\right)$ is an algebraic integer. However $2b$ and $uv$ are relatively prime, hence $\lambda_{k}\cos\left(\frac{k\theta}{n}\right)+\lambda_{n-k}\cos\left(\frac{(n-k)\theta}{n}\right)$ is an algebraic integer. As $\lambda_0$ is an integer, it follows that 
\[
\lambda_0+\sum_{k=1}^{n-1}\lambda_k\cos\left(\frac{k\theta}{n}\right)
\]
is an algebraic integer.

Reciprocally assume that $\lambda_0+\sum_{k=1}^{n-1}\lambda_k\cos\left(\frac{k\theta}{n}\right)$ is an algebraic integer. The following equality holds:
\[
\lambda_0+\sum_{k=1}^{n-1}\lambda_k\cos\left(\frac{k\theta}{n}\right) = \lambda_0+\sum_{k=1}^{n-1}\frac{\lambda_k}{2}\left(e^{\frac{ik\theta}{n}}+e^{\frac{-ik\theta}{n}}\right)=\lambda_0+\sum_{k=1}^{n-1}\frac{\lambda_k}{2}\left(\alpha^{\frac{k}{n}}+\alpha^{\frac{-k}{n}}\right).
\]
By Lemma~\ref{L:existencevaluations} there are valuations $\nu_{1, 2}$ and $\nu_{2, 2}$ defined over $\Q ( \alpha )$, extending the $2$-adic valuation, such that $\nu_{1, 2}( \alpha )=-\nu_{2, 2}( \alpha )=v_2( b )-1$. However, $v_2 ( b )-1$ is relatively prime to $n$. Moreover, for all $x \in \Q ( \alpha )^\ast$, we have that $\nu_{1, 2} ( x )$ is an integer. In fact since $4$ divides $b$ and $a$ and $b$ are coprime, by the reciprocity quadratic law, $( 2 )$ is not ramified over $\Q ( \alpha )$.  Hence it follows from Lemma~\ref{L:dim} that $[\Q(\alpha^{\frac{1}{n}}):\Q(\alpha)]=n$.

As $\nu_{1, 2}(\alpha)=v_2(b)-1$ is relatively prime to $n$, it follows from Lemma~\ref{L:groslambda} that $v_2 ( \lambda_0 )\ge 0$, and the following inequalities hold
\begin{equation*}
v_2\left(\frac{\lambda_k}{2}\right)\ge\Ceil{\frac{k(v_2(b)-1)}{n}}\,,\quad\text{for all $1\le k\le n$.}
\end{equation*}
Therefore $(3)$ holds.

Let $p$ be an odd prime dividing $b$. Thus by Lemma~\ref{L:existencevaluations}, there are valuations $\nu_{1, p}$ and $\nu_{2, p}$ extending $v_p$ to $\Q(\alpha)$ such that $\nu_{1, p} ( \alpha )=-\nu_{2, p} ( \alpha )=v_p ( b )$. As $v_p(b)$ is relatively prime to $n$, it follows from Lemma~\ref{L:groslambda} that $v_p(\lambda_0)\ge 0$, and the following inequalities hold
\begin{equation*}
v_p(\lambda_k)=v_p\left(\frac{\lambda_k}{2}\right)\ge\Ceil{\frac{kv_p(b)}{n}}\,,\quad\text{for all $1\le k\le n$.}
\end{equation*}
Therefore (2) holds.

It follows from Lemma~\ref{L:entieralgebrique} that $n\lambda_0$ is an integer, however we also have $v_p(\lambda_0)\ge 0$ for all odd primes $p$ dividing $n$, and $v_2(\lambda_0)\ge 0$, therefore $\lambda_0$ is an integer. Therefore $(1)$ holds.

By Lemma~\ref{L:entieralgebrique}, $nb\frac{\lambda_k}{2}+n\frac{\lambda_{n-k}}{2}(a+i\sqrt{b^2-a^2})$ is an algebraic integer (for all $1\le k\le n-1$). In particular, if $1\le k\le n-1$, then also $1\le n-k\le n-1$, and so $nb\frac{\lambda_{n-k}}{2}+n\frac{\lambda_k}{2}(a+i\sqrt{b^2-a^2})$ is an algebraic integer.

Therefore, applying $\Tr_{\Q(\alpha)/\Q}$, we obtain that $n(b\lambda_{k}+a\lambda_{n-k})$, $n(b\lambda_{n-k}+a\lambda_{k})$ are integers, and $n\lambda_{k}i\sqrt{b^2-a^2}$ is an algebraic integer.

We already proved that $v_2(\lambda_k)\ge 0$, and $v_2(\lambda_{n-k})\ge 0$, hence $v_2(u(\lambda_k+\lambda_{n-k}))\ge 0$, and $v_2(v(\lambda_k-\lambda_{n-k}))\ge 0$.

Similarly if $p$ be an odd prime number that divides $n$, then we have already proved that $v_p(\lambda_k)\ge 0$, and $v_p(\lambda_{n-k})\ge 0$, hence $v_p(u(\lambda_k+\lambda_{n-k}))\ge 0$, and $v_p(v(\lambda_k-\lambda_{n-k}))\ge 0$.

Let $p$ be a prime number that does not divides $n(b^2-a^2)$. As $n\lambda_{k}i\sqrt{b^2-a^2}$ is an algebraic integer, it follows that $n\lambda_{k}(b^2-a^2)$ is an integer, hence $v_p(\lambda_k)=v_p(n\lambda_{k}(b^2-a^2))\ge 0$. Similarly $v_p(\lambda_{n-k})\ge 0$, therefore $v_p(u(\lambda_k+\lambda_{n-k}))\ge 0$, and $v_p(v(\lambda_k-\lambda_{n-k}))\ge 0$.

Let $p$ be an odd prime number dividing $b^2-a^2=(b-a)(b+a)$. As $b$ and $a$ are relatively prime, it follows that either $p$ divides $b-a$ or $p$ divides $b+a$. Moreover $p$ does not divide $b$, and so $p$ does not divides $n$.

Assume that $p$ divides $b-a$. As $n(b\lambda_{k}+a\lambda_{n-k})$, $n(b\lambda_{n-k}+a\lambda_{k})$ are integers, it follows that $n(b\lambda_{k}+a\lambda_{n-k})+n(b\lambda_{n-k}+a\lambda_{k})=n(\lambda_k+\lambda_{n-k)})(a+b)$ is an integer. As $p$ does not divide $n$ and $p$ does not divide $a+b$, it follows that $v_p(u(\lambda_k+\lambda_{n-k})=v_p(n(\lambda_k+\lambda_{n-k)})(a+b))\ge 0$.

Moreover $v_p(\lambda_k)=v_p(n\lambda_{k}i\sqrt{b^2-a^2})-v_p(\sqrt{b^2-a^2})\ge 0-\frac{1}{2}v_p(b^2-a^2)=-\frac{1}{2}v_p(b-a)$. However, $\frac{b-a}{v^2}$ is a square-free integer, thus $v_p ( v )=\floor{\frac{1}{2}v_p ( b-a )}$. Therefore $v_p ( v )>\frac{1}{2}v_p ( b-a )-1$. It follows that $v_p(v\lambda_k)=v_p(v)+v_p(\lambda_k)> \frac{1}{2}v_p(b-a)-1 - \frac{1}{2}v_p(b-a)=-1$. However $v_p(v\lambda_k)$ is an integer, thus $v_p(v\lambda_k)\ge 0$. Similarly $v_p(v\lambda_{n-k})\ge 0$, and so $v_p(v(\lambda_k-\lambda_{n-k}))\ge 0$.

With a similar argument, we can prove that if $p$ divides $b+a$ then $v_p(u(\lambda_k+\lambda_{n-k}))\ge 0$ and $v_p(v(\lambda_k-\lambda_{n-k}))\ge 0$.

Hence, for all prime $p$ we have $v_p(u(\lambda_k+\lambda_{n-k}))\ge 0$ and $v_p(v(\lambda_k-\lambda_{n-k}))\ge 0$. Therefore $u(\lambda_k+\lambda_{n-k})$ and $v(\lambda_k-\lambda_{n-k})$ are integers. That is $(4)$ holds.
\end{proof}

\section{Trace in general}\label{sec3}

Given an algebraic number $\beta$ in a number field $K$, we denote $T(\beta)=\frac{\Tr_{K/\Q}(\beta)}{[K:\Q]}$. Note that $T$ does not depends on the choice of $K$. In particular $T(\beta)=\frac{\Tr_{\Q(\beta)/\Q}(\beta)}{[\Q(\beta):\Q]}$. Also note that $T(q)=q$ for all $q\in\Q$. We also denote by $\dc(\beta)$ the largest difference of two conjugates of $\beta$. Note that $\dc(\beta+q)=\dc(\beta)$ for all $q\in\Q$. The following Lemma is evident.

\begin{lemma}\label{L:trace0}
Let $\gamma$ be a totally real algebraic number. If $T ( \gamma )=0$, then $\gamma$ has a conjugate $\geq 0$ and a conjugate (maybe the same) $\leq 0$.
\end{lemma}

\begin{lemma}\label{L:trace1}
Let $\beta$ be a totally real algebraic number. Then $\beta$ has a conjugate $\geq T(\beta)$ and a conjugate (maybe the same) $\leq T(\beta)$.
\end{lemma}

\begin{proof}
Set $\gamma=\beta-T(\beta)$. 
As $T(\gamma)=0$ it follows from Lemma~\ref{L:trace0} that there is $\varphi\colon \Q(\beta)\to\R$ such that $\varphi(\gamma)\ge 0$, hence $\varphi(\beta) = \varphi(\gamma+T(\beta))=\varphi(\gamma)+T(\beta)\ge T(\beta)$. Similarly $\beta$ has a conjugate smaller than $T(\beta)$.
\end{proof}

\begin{lemma}\label{L:trace2}
Let $\beta$ be a totally real algebraic number. Then $\beta$ has a conjugate $\geq \sqrt{T(\beta^2)}$, or a conjugate $\leq -\sqrt{T(\beta^2)}$.
\end{lemma}

\begin{proof}
Lemma~\ref{L:trace1} implies that $\beta^2$ has a conjugate $\geq T(\beta^2)$. Hence we have a conjugate $\gamma$ of $\beta$ such that $\gamma^2\ge T(\beta^2)$, therefore $\gamma\ge \sqrt{T(\beta^2)}$ or $\gamma\le -\sqrt{T(\beta^2)}$.
\end{proof}

\begin{lemma}\label{L:trace3}
Let $\beta$ be a totally real algebraic number such that $T(\beta)=0$. Then $\dc(\beta)\ge 2\sqrt{T(\beta^2)}$.
\end{lemma}

\begin{proof}
We can assume that $\beta\not=0$. We see from Lemma~\ref{L:trace1} that $\beta^2$ has a conjugate $\geq T(\beta^2)$. Therefore $\beta$ has a conjugate $\geq \sqrt{T(\beta^2)}$ or a conjugate $\leq -\sqrt{T(\beta^2)}$. In the second case, we can change $\beta$ to $-\beta$, and so assume that $\beta$ has a conjugate $\geq \sqrt{T(\beta^2)}$.

Let $\gamma_1$ be the largest conjugate of $\beta$, and $\gamma_2$ be the smallest conjugate of $\beta$. From Lemma~\ref{L:trace1} it follows that $\gamma_1>0$, and $\gamma_2<0$. Let $u\in\Q$ be such that $u> \gamma_1$. Set $v=\frac{u^2-T(\beta^2)}{2u}$. As $u>\gamma_1\ge \sqrt{T(\beta^2)}$, it follows that $v>0$.

The following equalities hold:
\begin{align*}
T( (\beta-v)^2) &= T(\beta^2)-2vT(\beta)+v^2\\
&=T(\beta^2)+v^2 &&\text{as $T(\beta)=0$.}\\
&=T(\beta^2) + \left(\frac{u^2-T(\beta^2)}{2u}\right)^2\\
&=T(\beta^2) + \frac{u^4-2u^2T(\beta^2)+T(\beta^2)^2}{4u^2}\\
&=\frac{u^4+2u^2T(\beta^2)+T(\beta^2)^2}{4u^2}\\
&=\left(\frac{ u^2+T(\beta^2)}{2u}\right)^2.
\end{align*}
It follows from Lemma~\ref{L:trace2} that $\beta-v$ has a conjugate $\geq \frac{u^2+T(\beta^2)}{2u}$ or a conjugate $\leq -\frac{u^2+T(\beta^2)}{2u}$.

If $\beta-v$ has a conjugate $\geq \frac{u^2+T(\beta^2)}{2u}$, it follows that $\gamma_1$ the largest conjugate of $\beta$ satisfies
\[
\gamma_1 \geq v+\frac{u^2+T(\beta^2)}{2u} = \frac{u^2-T(\beta^2)}{2u}+\frac{u^2+T(\beta^2)}{2u}=u\,,
\]
which contradicts $u>\gamma_1$. Therefore $\beta-v$ has a conjugate $\leq -\frac{u^2+T(\beta^2)}{2u}$. It follows that $\gamma_2$ the smallest conjugate of $\beta$ satisfies
\[
\gamma_2 \leq v-\frac{u^2+T(\beta^2)}{2u} = \frac{u^2-T(\beta^2)}{2u}-\frac{u^2+T(\beta^2)}{2u}=-\frac{T(\beta^2)}{u}\,.
\]
This inequality is true for all $u\in\Q$ larger than $\gamma_1$. It follows that $\gamma_2\le-\frac{T(\beta^2)}{\gamma_1}$. Thus $\gamma_1-\gamma_2\ge \gamma_1 + \frac{T(\beta^2)}{\gamma_1}\ge 2\sqrt{T(\beta^2)}$.
\end{proof}

\begin{lemma}\label{L:trace4}
Let $\beta$ be a totally real algebraic number, set $r=T\left((\beta^2-T(\beta^2))^2\right)$. If $T(\beta)=0$ and $T(\beta^3)=0$ then $\dc(\beta)\ge 2\sqrt{2\sqrt{r}}$.
\end{lemma}

\begin{proof}
Denote by $\gamma_1$ the largest conjugate of $\beta$ and by $\gamma_2$ the smallest conjugate of $\beta$. If $\module{\gamma_2}\ge\module{\gamma_1}$, then we can change $\beta$ to $-\beta$, and so assume that $\gamma_1=\module{\gamma_1}\ge\module{\gamma_2}$, in particular $\gamma_1^2$ is the largest conjugate of $\beta^2$. Note that $T(\beta^2-T(\beta^2))=0$, hence it follows from Lemma~\ref{L:trace3} that $\dc(\beta^2-T(\beta^2))\ge 2\sqrt{T\left((\beta^2-T(\beta^2))^2\right)}=2\sqrt{r}$. However all conjugates of $\beta^2$ are positives, hence $\gamma_1^2\ge\dc(\beta^2)=\dc(\beta^2-T(\beta^2))\ge 2\sqrt{r}$, it follows that $\gamma_1\ge \sqrt{2\sqrt{r}}$.

Consider the following function
\begin{align*}
f\colon\R&\to\R\\
x &\mapsto x-\gamma_1+\sqrt{2\sqrt{r+4x^2 T(\beta^2)}}.
\end{align*}
Note that $f(0)= -\gamma_1 + \sqrt{2\sqrt{r}}\le 0$ and that $\lim_{x\to\infty} f(x)=+\infty$. Moreover $f$ is continuous, therefore we can pick $c\ge 0$, such that $f(c)=0$. Set $s=r+4c^2 T(\beta^2)$. Hence
\begin{equation}\label{E:t4:defc}
c = \gamma_1-\sqrt{2\sqrt{r+4c^2 T(\beta^2)}}=\gamma_1 - \sqrt{2\sqrt{s}}.
\end{equation}
Let $d\in \Q$ be such that $d>c$. Set $\rho=(\beta-d)^2-T((\beta-d)^2)$. Note that $T(\rho)=0$, moreover the following equalities hold
\begin{align*}
\rho&=\beta^2-2d\beta+d^2-T(\beta^2-2d\beta+d^2)\\
&=\beta^2-2d\beta+d^2-(T(\beta^2)+2dT(\beta)+d^2)\\
&=\beta^2-T(\beta^2) -2d\beta.
\end{align*}
Thus
\begin{align*}
T(\rho^2)&=T\left( (\beta^2-T(\beta^2) -2d\beta)^2\right)\\
&= T\left( (\beta^2-T(\beta^2))^2\right) +4d^2T(\beta^2) - 4dT(\beta^3)+4dT(\beta^2)T(\beta)\\
&= r + 4d^2T(\beta^2)\\
&\ge s.
\end{align*}
Therefore, by Lemma~\ref{L:trace3}, we have $\dc(\rho)\ge 2\sqrt{s}$, thus $\dc((\beta-d)^2)=\dc(\rho)\ge 2\sqrt{s}$. However all conjugates of $(\beta-d)^2$ are positives, hence there is a conjugate $\gamma_3$ of $\beta$ such that $(\gamma_3-d)^2\ge 2\sqrt{s}$. Assume that $\gamma_3-d\ge\sqrt{2\sqrt{s}}$, then, as $d>c$ it follows from \eqref{E:t4:defc} that $\gamma_3\ge d+\sqrt{2\sqrt{s}}>c+\sqrt{2\sqrt{s}}=\gamma_1$; a contradiction, as $\gamma_1$ is the largest conjugate of $\beta$. Therefore $\gamma_3-d\le -\sqrt{2\sqrt{s}}$, hence $\gamma_2\le\gamma_3\le d-\sqrt{2\sqrt{s}}$. Therefore, for all $d\in\Q$ such that $d>c$, we have $\gamma_2\le d-\sqrt{2\sqrt{s}}$. Hence $\gamma_2\le c-\sqrt{2\sqrt{s}}=\gamma_1-2\sqrt{2\sqrt{s}}$, thus $\gamma_1-\gamma_2\ge \gamma_1 - (\gamma_1-2\sqrt{2\sqrt{s}})= 2\sqrt{2\sqrt{s}}\ge 2\sqrt{2\sqrt{r}}$.
\end{proof}

\section{Trace in the fields for the examples}\label{sec4}

In this section we fix $0<a<b$ relatively prime integers. We assume that if $2$ divides $b$, then $4$ divides $b$. Denote $\theta=\arccos(\frac{a}{b})$. We consider $R$ a set of positive integers such that for all $n,m\in R$ the following statements hold
\begin{itemize}
\item $\gcd(m,n)\in R$.
\item For all $s$ that divides $n$ we have $s\in R$.
\item Let $p$ be an odd prime that divides $b$, then $v_p(b)$ is relatively prime to $n$.
\item If $b$ is even then $v_2(b)-1$ is relatively prime to $n$.
\end{itemize}
Denote $Q=\Setm{\frac{k}{n}}{n\in R\text{ and } 1\le k\le n}$. Note that it follows from Corollary~\ref{C:TraceNul} that
\begin{equation*}
T\left(\cos\left(\frac{k\theta}{n}\right)\right)=0\,,\text{for all $n\in R$, and all $1\le k\le n$.}
\end{equation*}

For all $x,y\in\Q$ we set $N(x,y) = x^2+y^2+2xy\frac{a}{b}$.

\begin{lemma}\label{L:TSommeCos}
Let $n$ be in $R$. Let $\famm{\lambda_k}{0\le k\le n-1}$ be a family in $\mathbb{Q}$. Then
\begin{align*}
T\left( \left(\lambda_0+\sum_{k=1}^{n-1}\lambda_k\cos\left(\frac{k\theta}{n}\right)\right)^2\right)&=\lambda_0^2+\frac{1}{2}\sum_{k=1}^{n-1} \lambda_k^2 + \lambda_k\lambda_{n-k}\frac{a}{b}\\
&=\lambda_0^2+\frac{1}{2}\sum_{k=1}^{\floor{\frac{n}{2}}} N(\lambda_k,\lambda_{n-k}) \,.
\end{align*}
\end{lemma}

\begin{proof}
Set $\alpha=e^{i\theta}$. Set $K=\mathbb{Q}(e^{\frac{i\theta}{n}})$, note that for every $1\le k,\ell\le n-1$, the following equalities hold:
\begin{align*}
\cos\left(\frac{k\theta}{n}\right)\cos\left(\frac{\ell\theta}{n}\right)
&=\frac{1}{4}\left(e^{\frac{ki\theta}{n}}+e^{\frac{-ki\theta}{n}}\right)\left(e^{\frac{\ell i\theta}{n}}+e^{\frac{-\ell i\theta}{n}}\right)\\
&=\frac{1}{4}\left( e^{\frac{(k+\ell)i\theta}{n}}+e^{\frac{(k-\ell)i\theta}{n}} +e^{\frac{(-k+\ell)i\theta}{n}}+e^{\frac{(-k-\ell)i\theta}{n}} \right).
\end{align*}
Therefore,
\[
\Tr_{\Q(\alpha^{\frac{1}{n}})/\Q(\alpha)}\left(\cos\left(\frac{k\theta}{n}\right)\cos\left(\frac{\ell\theta}{n}\right)\right)=
\begin{cases}
\frac{n}{2}+\frac{na}{2b} &\text{if $k=\ell=\frac{n}{2}$}\\
\frac{n}{2} &\text{if $k=\ell\not=\frac{n}{2}$}\\
\frac{na}{2b} &\text{if $k=n-\ell\not=\frac{n}{2}$}\\
0 &\text{otherwise}.
\end{cases}
\]
Hence:
\[
\Tr_{\Q(\cos(\frac{\theta}{n}))/\Q}\left(\cos\left(\frac{k\theta}{n}\right)\cos\left(\frac{\ell\theta}{n}\right)\right)=
\begin{cases}
\frac{n}{2}+\frac{na}{2b} &\text{if $k=\ell=\frac{n}{2}$}\\
\frac{n}{2} &\text{if $k=\ell\not=\frac{n}{2}$}\\
\frac{na}{2b} &\text{if $k=n-\ell\not=\frac{n}{2}$}\\
0 &\text{otherwise}.
\end{cases}
\]
Moreover $\Tr_{\Q(\cos(\frac{\theta}{n}))/\Q}(\cos(\frac{k\theta}{n}))=0$ for all $1\le k\le n-1$. The Lemma follows directly.
\end{proof}

\begin{remark}
Given $\lambda,\mu\in\Q$, $n\in R$, and $1\le k<\frac{n}{2}$. It follows from Lemma~\ref{L:TSommeCos} that
\[
N(\lambda,\mu)=2T\left(\left( \lambda\cos\left(\frac{k\theta}{n}\right) + \mu\cos\left(\frac{(n-k)\theta}{n}\right) \right)^2\right)\,.
\]
Moreover if $2\in R$, then 
\[
N(\lambda,\lambda)=2T\left( \left( \lambda\cos\left(\frac{\theta}{2}\right) \right)^2\right)\,.
\]
\end{remark}

\section{Some analysis}\label{sec5}

In this section we give various estimation of maximal value of functions that will be necessary to the calculation of the Julia Robinson Number in the next section.

\begin{lemma}\label{L:maxsommecossin}
Let $\lambda,\mu$ be real numbers. Then the maximum value, with $x\in\R$, of $\lambda\cos x + \mu\sin x$ is $\sqrt{\lambda^2+\mu^2}$. The minimum value is $-\sqrt{\lambda^2+\mu^2}$.
\end{lemma}

\begin{proof}
A rotation of the angle $x$, of the vector $(\lambda,\mu)$, has first coordinate $\lambda\cos x + \mu\sin x$. Hence the maximum possible value is the length of the vector $(\lambda,\mu)$, which is $\sqrt{\lambda^2+\mu^2}$.
\end{proof}

\begin{lemma}\label{L:maxsommecos}
Let $\lambda,\mu,\theta,q$ be real numbers. Then the maximum value, for $x\in\R$, of $\lambda\cos(q\theta+x)+\mu\cos( (1-q)\theta-x)$ is $\sqrt{\lambda^2+\mu^2+2\lambda\mu\cos(\theta)}$.
\end{lemma}

\begin{proof}
For all $x\in\R$, denote $f(x)=\lambda\cos(q\theta+x)+\mu\cos( (1-q)\theta-x)$. It follows, from Lemma~\ref{L:sommecos}, that
\begin{equation*}
f(x) = (\lambda+\mu)\cos\left(\frac{\theta}{2}\right)\cos\left(\frac{1-2q}{2}\theta + x\right) + (\mu-\lambda)\sin\left(\frac{\theta}{2}\right)\sin\left(\frac{1-2q}{2}\theta + x\right)
\end{equation*}
Therefore, by Lemma~\ref{L:maxsommecossin}, the maximum value of $f$ is
\[
V=\sqrt{ (\lambda+\mu)^2\cos^2\left(\frac{\theta}{2}\right) + (\mu-\lambda)^2\sin^2\left(\frac{\theta}{2}\right)}\,.
\]
The following equalities hold
\begin{align*}
V^2&=(\lambda^2+\mu^2)\left(\cos^2\left(\frac{\theta}{2}\right) + \sin^2\left(\frac{\theta}{2}\right)\right) +2\lambda\mu\left(\cos^2\left(\frac{\theta}{2}\right) - \sin^2\left(\frac{\theta}{2}\right)\right)\\
& = \lambda^2+\mu^2 + 2\lambda\mu\cos(\theta)\,.
\end{align*}
Therefore the maximum value is $V=\sqrt{\lambda^2+\mu^2 + 2\lambda\mu\cos(\theta)}$.
\end{proof}

\begin{corollary}\label{C:grandconjuguecassimple}
Let $\lambda,\mu$ be real numbers, and $\varepsilon>0$. Set 
\[
A=\max\left(1,\Ceil{2\pi\frac{\module{\lambda}+\module{\mu}}{\varepsilon}}\right)\,.
\]
Let $\theta$ be a real number. Set
\[
s=\lambda^2+\mu^2+2\lambda\mu\cos(\theta)\,.
\]
Then for all $n\ge A$, and all $1\le k\le n-1$ with $\gcd(k,n)=1$, there is $t\in\Z$ such that
\[
\lambda\cos\frac{k(\theta+2\pi t)}{n}+\mu\cos\frac{(n-k)(\theta+2\pi t)}{n}\ge\sqrt{s}-\varepsilon\,.
\]
\end{corollary}

\begin{proof}
For each $t\in\Z$, we set
\begin{align*}
\alpha_t &= \lambda\cos\left( \frac{k(\theta+2\pi t)}{n}\right)+\mu\cos\left(\frac{(n-k)(\theta+2\pi t)}{n}\right)\\
&=\lambda\cos\left( \frac{k\theta}{n}+\frac{2\pi kt}{n}\right)+\mu\cos\left(\frac{(n-k)\theta}{n}-\frac{2\pi kt}{n}\right).
\end{align*}
Let $r$ be an integer such that $rk\equiv 1\mod n$. By Lemma~\ref{L:maxsommecos}, we can pick $x\in\R$ such that
\begin{equation}\label{E:cassimplesomecos1}
\sqrt{s}=\lambda\cos\left( \frac{k\theta}{n} + x\right)+\mu\cos\left(\frac{(n-k)\theta}{n} - x\right)\,.
\end{equation}
Set $q=\floor{\frac{nx}{2\pi}}$. Set $t=rq$, hence $kt=krq\equiv q\mod n$. The following equalities hold
\begin{equation}\label{E:cassimplesomecos2}
\alpha_t  =\lambda\cos\left( \frac{k\theta}{n}+\frac{2\pi q}{n}\right)+\mu\cos\left(\frac{(n-k)\theta}{n}-\frac{2\pi q}{n}\right)\,.
\end{equation}
The following inequalities hold
\begin{equation}\label{E:cassimplesomecos3}
0\le x-\frac{2\pi q}{n} = x - \frac{2\pi}{n}\Floor{\frac{nx}{2\pi}}\le \frac{2\pi}{n}\le \frac{\varepsilon}{\module{\lambda} + \module{\mu}}\,.
\end{equation}
However, $\module{\cos(a)-\cos(b)}\le \module{a-b}$ for all $a,b\in\R$. Therefore, subtracting \eqref{E:cassimplesomecos1} and \eqref{E:cassimplesomecos2}, and comparing the argument of the cosine function in \eqref{E:cassimplesomecos3}, we obtain
\[
\module{\alpha_t-\sqrt{s}}\le (\module{\lambda}+\module{\mu})\frac{\varepsilon}{\module{\lambda}+\module{\mu}} = \varepsilon\,.
\]
It follows that $\alpha_t\ge \sqrt{s}-\varepsilon$.
\end{proof}

\begin{lemma}\label{L:somecoszero}
Let $\gamma$ be a real number, $\ell\not=0$ be an integer, and $d>1$ be an integer. Assume that $d\notdiv\ell$. Then $\sum_{t=0}^{d-1} \cos\left(\gamma+\frac{t\ell}{d}2\pi\right)=0$.
\end{lemma}

\begin{proof}
Set $x=e^{i\frac{\ell}{d}2\pi}$. Note that $x^d=1$. The following equalities hold:
\begin{align*}
\sum_{t=0}^{d-1} \cos\left(\gamma+\frac{t}{d}2\pi\right) &= \Re\sum_{t=0}^{d-1} \exp \left(i\gamma+i\frac{t\ell}{d}2\pi\right)\\
&=\Re\left(e^{i\gamma}\sum_{t=0}^{d-1} x^t\right)\\
&=\Re\left(e^{i\gamma}\frac{x^d-1}{x-1}\right)\\
&=0\tag*{\qed}
\end{align*}
\renewcommand{\qed}{}
\end{proof}

\begin{notation}
For $a\in\R$ and $b\in\R^*$, we denote $\modulo{a}{b}$ the only real number $c\in[-\frac{b}{2},\frac{b}{2}[$ such that $a-c\in b\Z$.
Note that for all $t\in\R^*$ we have $\modulo{ta}{tb}=t \modulo{a}{b}$. We will use the particular case $\modulo{\frac{a}{b} 2\pi}{2\pi}=\frac{\modulo{a}{b}}{b}2\pi$

Note that if $a,b$ are integers, than $R_b(a)$ is an integer. Also observe that if $a,a',b$ are integers, then $\module{\modulo{a a'}{b}}\le\module{\modulo{a}{b}}\times \module{\modulo{a'}{b}}$.
\end{notation}

\begin{lemma}\label{L:somecospositive}
Let $\lambda,\mu,\theta_1,\theta_2,\rho$ be real numbers, and $u$ be an integer. If $\modulo{u\rho}{2\pi}\not=u\modulo{\rho}{2\pi}$ and $\module{\modulo{u\rho}{2\pi}}<\frac{\pi}{2}$, then there is $\module{r}\le\max(1,\module{u}-1)$ such that $\lambda\cos(\theta_1+r\rho)+\mu\cos( \theta_2-r\rho)\ge 0$.
\end{lemma}

\begin{proof}
We cannot have $u=0$ or $u=1$. Assume that $u=-1$. It follows, from $\modulo{u\rho}{2\pi}\not=u\modulo{\rho}{2\pi}$, that $\rho=\pi$. Hence $\lambda\cos(\theta_1+\rho)+\mu\cos( \theta_2-\rho) = -(\lambda\cos(\theta_1)+\mu\cos( \theta_2))$. Hence $r=0$, or $r=1$ satisfies the required condition.

We now assume that $u\ge 2$. We can assume that $\modulo{\rho}{2\pi}=\rho$, that is $\rho\in[-\pi,\pi[$. Set $\lambda'=\lambda\cos(\theta_1) + \mu\cos(\theta_2)$, and $\mu'=\mu\sin(\theta_2)-\lambda\sin(\theta_1)$. Given $r\in\Z$, denote $f(r) = \lambda\cos(\theta_1+r\rho)+\mu\cos( \theta_2-r\rho)$, hence the following equalities hold
\begin{align*}
f(r)&= \lambda(\cos(\theta_1)\cos(r\rho)-\sin(\theta_1)\sin(r\rho)) + \mu(\cos(\theta_2)\cos(r\rho)+\sin(\theta_2)\sin(r\rho))\\
&= \lambda'\cos(r\rho) + \mu'\sin(r\rho).
\end{align*}
If $\lambda'\ge 0$, then $f(0)=\lambda'\ge 0$, as required. Assume that $\lambda'<0$. We first want to find $s$ such that $\modulo{s\rho}{2\pi}=s\rho$ and $s\rho\not \in[-\frac{\pi}{2},\frac{\pi}{2}[$.

If $\rho \not\in[-\frac{\pi}{2},\frac{\pi}{2}[$, then $s=1$ satisfies the desired conditions. Assume that $\rho \in[-\frac{\pi}{2},\frac{\pi}{2}[$. As $u\rho\not\in[-\pi,\pi[$, we can take the smallest $s\ge 1$ such that  $s\rho\not\in[-\frac{\pi}{2},\frac{\pi}{2}[$. From the minimality we have $(s-1)\rho\in[-\frac{\pi}{2},\frac{\pi}{2}[$, however $\rho \in[-\frac{\pi}{2},\frac{\pi}{2}[$, hence $s\rho\in[-\pi,\pi[$. Thus $\modulo{s\rho}{2\pi}=s\rho$, in particular $\module{s}\le\module{u}-1$.

Note that $\cos(s\rho)=\cos(-s\rho)\le 0$. Moreover $\sin(s\rho)=-\sin(-s\rho)$. Hence $f(s)=\lambda'\cos(s\rho)+\mu'\sin(s\rho)\ge 0$ or $f(-s)=\lambda'\cos(s\rho)-\mu'\sin(s\rho)\ge 0$.

The proof is similarly if $u\le -2$.
\end{proof}

\begin{lemma}\label{L:smallcongruence}
Let $a\in\Z$ and $b>0$ an integer, then there exists $u\in\Z\setminus\set{0}$ such that $\module{u}\le\sqrt{b}$ and $\module{\modulo{u}{b}}\le\sqrt{b}$.
\end{lemma}

\begin{proof}
If $b=1$ the it is easy. Assume that $b\ge 2$. Consider the following system with indeterminate $u,v$ in $\R$.
\begin{equation*}
\begin{cases}
\module{u}\le\sqrt{b}\\
\module{ua+vb}\le\sqrt{b}.
\end{cases}
\end{equation*}
The set of solutions is a parallelogram of height $2\sqrt{b}$, and base $2\frac{\sqrt{b}}{b}$. Hence the surface is $4$. Therefore, by Minkowski's Theorem (see for instance \cite[Lemma, p. 137] {Marcus}), there is a solution $(u,v)\in\Z^2\setminus\set{(0,0)}$.

Thus we have $\module{u}\le\sqrt{b}$ and $\module{\modulo{ua}{b}}\le \module{ua+vb}\le \sqrt{b}$. Moreover if $u=0$, then $v\not=0$ and $b\le \module{vb}=\module{ua+vb}\le\sqrt{b}$; a contradiction. Therefore $u\not=0$ as required.
\end{proof}

\begin{lemma}\label{L:enumerationcastheoremanalyse}
Let $1\le k,\ell<n$ be integers such that $\gcd(k,\ell,n)=1$, $k\not=\ell$, and $k\not=n-\ell$. Then one of the following statement holds.
\begin{enumerate}
\item $1<\gcd(k,n)\le\sqrt{n}$ or $1<\gcd(\ell,n)\le\sqrt{n}$.
\item There are non-zero integers $r,u$ such that $\module{u}\le\sqrt{n}$, $rk\equiv 1\mod n$,
$u(\modulo{r\ell}{n})\not=\modulo{ur\ell}{n}$, and $\modulo{ur\ell}{n}\le\sqrt{n}$.
\item There are non-zero integers $r,u$ such that $\module{u}\le\sqrt{n}$, $r\ell\equiv 1\mod n$,
$u(\modulo{rk}{n})\not=\modulo{urk}{n}$, and $\modulo{urk}{n}\le\sqrt{n}$.
\end{enumerate}
\end{lemma}

\begin{proof}
Note that $k,\ell,n-k,n-\ell$ are all distinct in $\set{1,\dots,n}$, thus $n\ge 5$. As $\gcd(k,\ell,n)=1$, it follows that $\gcd(\ell,n)\times\gcd(k,n)$ divides $n$

Assume that $\gcd(k,n)>\sqrt{n}$. Then $\gcd(\ell,n)\le\frac{n}{\gcd(k,n)}<\sqrt{n}$. If $\gcd(\ell,n)>1$ then $(1)$ holds. Assume that $\gcd(\ell,n)=1$. Let $r$ be an inverse of $\ell$ modulo $u$, thus $r\ell\equiv 1\mod n$. Set $u=\frac{n}{\gcd(k,n)}$, so $\module{u}=u\le \sqrt{n}$ and $\modulo{urk}{n} = \modulo{\frac{rk}{\gcd(k,n)}n} {n} = 0$. As $r\ell\equiv 1\mod n$ and $k<n$ it follows that $n\notdiv rk$, thus $\modulo{rk}{n}\not=0$. So $u(\modulo{rk}{n})\not=\modulo{urk}{n}$, therefore $(3)$ holds.

Similarly if $\gcd(\ell,n)>\sqrt{n}$, then $(1)$ or $(3)$ holds. Hence we can assume that $\gcd(k,n)\le \sqrt{n}$ and $\gcd(\ell,n)\le \sqrt{n}$. If $\gcd(k,n)>1$, or $\gcd(\ell,n)>1$ then $(1)$ holds. So we can assume that $\gcd(k,n)=1=\gcd(\ell,n)$.

Let $1\le r<n$ be the inverse of $k$ modulo $n$, and $1\le q<n$ be the inverse of $\ell$ modulo $n$. First assume that $\module{\modulo{r\ell}{n}}>\sqrt{n}$. It follows from Lemma~\ref{L:smallcongruence} that there is $u\in\Z\setminus\set{0}$ such that $\module{u}\le\sqrt{n}$ and $\module{\modulo{ur\ell}{n}}\le\sqrt{n}$. So $\module{u(\modulo{r\ell}{n})}=\module{u}\module{\modulo{r\ell}{n}}>\sqrt{n}\ge \module{\modulo{ur\ell}{n}}$, thus $(2)$ holds. Similarly $\module{\modulo{qk}{n}}>\sqrt{n}$ then $(3)$ holds. Therefore we can assume that $\module{\modulo{r\ell}{n}}\le\sqrt{n}$ and $\module{\modulo{qk}{n}}\le\sqrt{n}$.

Assume that $\module{\modulo{r\ell}{n}}\ge\sqrt{n-\sqrt{n}}$. Set $u=\modulo{r\ell}{n}$, note that $\module{u}\le\sqrt{n}$, moreover $n\ge u(\modulo{r\ell}{n}) = (\modulo{r\ell}{n})^2\ge \sqrt{n-\sqrt{n}}^2=n-\sqrt{n}$, and so $\module{\modulo{ur\ell}{n}}\le\sqrt{n}$. Moreover $u(\modulo{r\ell}{n})\ge n-\sqrt{n}>\sqrt{n}$, hence $u(\modulo{r\ell}{n})\not=\modulo{ur\ell}{n}$, therefore $(2)$ holds. Similarly if $\module{\modulo{r\ell}{n}}\ge\sqrt{n-\sqrt{n}}$ then $(3)$ holds.

We can assume that $\module{\modulo{r\ell}{n}}<\sqrt{n-\sqrt{n}}$ and $\module{\modulo{qk}{n}}<\sqrt{n-\sqrt{n}}$. Thus $\module{\modulo{r\ell}{n}}\times \module{\modulo{qk}{n}}<n-\sqrt{n}$. However, $rkq\ell\equiv 1\mod n$, hence $\module{\modulo{r\ell}{n}}\times \module{\modulo{qk}{n}}\equiv\pm 1\mod n$. Therefore $\module{\modulo{r\ell}{n}}\times \module{\modulo{qk}{n}}=1$. It follows that $\module{\modulo{r\ell}{n}}= \module{\modulo{qk}{n}}=1$, so that $r\ell\equiv \pm 1\mod n$, thus $k\equiv kr\ell\equiv \pm\ell\mod n$. Therefore $k=\ell$ or $k=n-\ell$; a contradiction.
\end{proof}

\begin{theorem}\label{T:conjuguecascomplexe}
Let $\lambda_1,\mu_1,\lambda_2,\mu_2$ be real numbers and $\varepsilon>0$. Set 
\[
A=\max\left(17,\Ceil{2\pi\frac{\module{\lambda_1}+\module{\lambda_2}+\module{\mu_1}+\module{\mu_2}}{\varepsilon}}^2\right)\,.
\]
Let $\theta_1, \theta_2$ be real numbers. Set
\[
s=\min(\lambda_1^2+\mu_1^2+2\lambda_1\mu_1\cos(\theta_1),\lambda_2^2+\mu_2^2+2\lambda_2\mu_2\cos(\theta_2))\,.
\]
Then for all $n\ge A$, and all $1\le k,\ell\le n-1$ with $k\not=\ell$, $k\not=n-\ell$, and $\gcd(k,\ell,n)=1$, there is $t\in\Z$ such that
\[
\lambda_1\cos\frac{k(\theta_1+2\pi t)}{n}+\mu_1\cos\frac{(n-k)(\theta_1+2\pi t)}{n}+\lambda_2\cos\frac{\ell(\theta_2+2\pi t)}{n}+\mu_2\cos\frac{(n-\ell)(\theta_2+2\pi t)}{n}>\sqrt{s}-\varepsilon\,.
\]
\end{theorem}

\begin{proof}
Let $x$ be as in Lemma~\ref{L:maxsommecos}, that is
\begin{equation}\label{E:xmaximum}
\lambda_1\cos\left(\frac{k\theta_1}{n}+x\right)+\mu_1\cos\left(\frac{(n-k)\theta_1}{n}-x\right) = \sqrt{\lambda_1^2+\mu_1^2+2\lambda_1\mu_1\cos(\theta)}\ge\sqrt{s}\,.
\end{equation}

Let $n$ be $\geq A$, let $1\le k,\ell\le n-1$ be with $k\not=\ell$, $k\not=n-\ell$, and $\gcd(k,\ell,n)=1$. For all $j\in\Z$, we denote
\[
\alpha_j = \lambda_1\cos\frac{k(\theta_1+2\pi j)}{n}+\mu_1\cos\frac{(n-k)(\theta_1+2\pi j)}{n}
\]
\[
\beta_j =\lambda_2\cos\frac{\ell(\theta_2+2\pi j)}{n}+\mu_2\cos\frac{(n-\ell)(\theta_2+2\pi j)}{n}.
\]

One of the statement of Lemma~\ref{L:enumerationcastheoremanalyse}(1), (2), or (3) holds. Assume that Lemma~\ref{L:enumerationcastheoremanalyse}(1) holds. By permuting $k$ and $\ell$, we can assume $1<\gcd(k,n)\le\sqrt{n}$. Set $d=\gcd(k,n)$. Let $r\in\Z$ such that $kr\equiv d\mod n$. Set $q=\floor{\frac{nx}{2\pi d}}$. The following equalities hold: 
\begin{align*}
\alpha_{rq} &= \lambda_1\cos\frac{k(\theta_1+2\pi rq)}{n}+\mu_1\cos\frac{(n-k)(\theta_1+2\pi rq)}{n} \\
&= \lambda_1\cos\left(\frac{k\theta_1}{n} +\frac{2\pi krq}{n}\right)+\mu_1\cos\left(\frac{(n-k)\theta_1}{n}-\frac{2\pi krq}{n}\right) \\
&= \lambda_1\cos\left(\frac{k\theta_1}{n} +\frac{2\pi dq}{n}\right)+\mu_1\cos\left(\frac{(n-k)\theta_1}{n}-\frac{2\pi dq}{n}\right). \\
\end{align*}
Note that the following inequalities hold:
\[
0\le x-\frac{2\pi dq}{n} = x - \frac{2\pi d}{n}\Floor{\frac{nx}{2\pi d}}\le \frac{2\pi d}{n}\le \frac{2\pi}{\sqrt{n}}\le\frac{2\pi}{\sqrt{A}}\le \frac{\varepsilon}{\module{\lambda_1} + \module{\mu_1}}\,.
\]
However, $\module{\cos(a)-\cos(b)}\le \module{a-b}$ for all $a,b\in\R$. It follows that
\[
\alpha_{rq} \ge    \lambda_1\cos\left(\frac{k\theta_1}{n} + x\right)+\mu_1\left(\cos\frac{(n-k)\theta_1}{n}-x\right)- (\module{\lambda_1}+\module{\mu_1})\frac{\varepsilon}{\module{\lambda_1}+\module{\mu_1}} = \sqrt{s}-\varepsilon\,.
\]
Set $r_1=\frac{n}{d}$ and $r_2=\frac{k}{d}$. So $kr_1=nr_2$. For all $t\in\Z$ the following equalities hold:
\begin{align*}
\alpha_{rq+tr_1} &=  \lambda_1\cos\frac{k(\theta_1+2(rq+tr_1)\pi)}{n}+\mu_1\cos\frac{(n-k)(\theta_1+2(rq+tr_1)\pi)}{n}\\
&= \lambda_1\cos\left(\frac{k(\theta_1+2\pi pq)}{n}+\frac{2\pi ktr_1}{n}\right)+\mu_1\cos\left(\frac{(n-k)(\theta_1+2\pi rq)}{n}-\frac{2\pi ktr_1}{n}\right)\\
&= \lambda_1\cos\left(\frac{k(\theta_1+2\pi rq)}{n}+2\pi tr_2\right)+\mu_1\cos\left(\frac{(n-k)(\theta_1+2pr\pi)}{n}-2\pi tr_2\right)\\
&=\alpha_{rq}.
\end{align*}
Set $\gamma_1 = \frac{\ell(\theta_2+2\pi rq)}{n}$ and $\gamma_2 = \frac{(n-\ell)(\theta_2+2\pi rq)}{n}$. For all $t\in\Z$ the following equalities hold:
\begin{align*}
\beta_{rq+tr_1} &= \lambda_2\cos\left(\frac{\ell(\theta_2+2\pi rq)}{n} + t\frac{2\pi r_1 \ell}{n}\right)+\mu_2\cos\left(\frac{(n-\ell)(\theta_2+2\pi rq)}{n} - t\frac{2\pi r_1 \ell}{n} \right) \\
&=\lambda_2\cos\left(\gamma_1+t\frac{\ell}{d}2\pi\right) + \mu_2\cos\left(\gamma_2-t\frac{\ell}{d}2\pi\right).
\end{align*}
Note that $\gcd(d,\ell)=1$. Then, we can deduce from Lemma~\ref{L:somecoszero} that $\sum_{t=0}^{d-1}\beta_{pr+tr_1} = 0$, hence there is $1\le t\le d-1$ such that $\beta_{pr+tr_1}\ge 0$. Therefore
\[
\alpha_{rq+tr_1}+\beta_{rq+tr_1} \ge \alpha_{rq}\ge \sqrt{s}-\varepsilon\,.
\]

Now assume that Lemma~\ref{L:enumerationcastheoremanalyse}(2) holds. Let $r,u$ be non-zero integers such that $\module{u}\le\sqrt{n}$, $rk\equiv 1\mod n$, $u\modulo{r\ell}{n}\not=\modulo{ur\ell}{n}$, and $\modulo{ur\ell}{n}\le\sqrt{n}$. Set $q=\floor{\frac{nx}{2\pi}}$. For each $t\in\Z$ we have 
\begin{equation*}
\alpha_{r ( q+t )} =  \lambda_1\cos\left(\frac{k\theta_1}{n}+\frac{2\pi q}{n}+\frac{2\pi t}{n} \right)+\mu_1\cos\left(\frac{(n-k)(\theta_1)}{n}-\frac{2\pi q}{n}-\frac{2\pi t}{n}\right)
\end{equation*}
If $\module{t}\le\module{u}-1$ then the following inequalities hold
\[
\module{\frac{2\pi q}{n} + \frac{2\pi t}{n} -x}\le\frac{2\pi}{n} + \frac{2\pi\module{t}}{n}\le \frac{2\pi\module{u}}{n}\le\frac{2\pi}{\sqrt{A}}\le \frac{\varepsilon}{\module{\lambda_1}+\module{\mu_1}}\,.
\]
As before it follows that
\begin{equation*}
\alpha_{r(q+t)} \ge \lambda_1\cos\left(\frac{k\theta_1}{n} + x\right)+\mu_1\left(\cos\frac{(n-k)\theta_1}{n}-x\right)- \varepsilon\,.
\end{equation*}
We obtain the following inequality
\begin{equation}\label{E:minorealpha}
\alpha_{r(q+t)} \ge \sqrt{s}-\varepsilon\,,\quad\text{for all $\module{t}\le\module{u}-1$}.
\end{equation}

Note that
\begin{equation*}
\beta_{r(q+t)} = \lambda_2\cos\left(\frac{\ell(\theta_2+2\pi rq)}{n} +t\frac{2\pi r\ell}{n}  \right)+\mu_2\left(\cos\frac{(n-\ell)(\theta_2+2\pi rq)}{n}-t\frac{2\pi r\ell}{n}\right). 
\end{equation*}
Set $\gamma_1= \frac{\ell(\theta_2+2\pi rq)}{n}$, $\gamma_2 = \frac{(n-\ell)(\theta_2+2\pi rq)}{n}$, and $\rho=\frac{2\pi r\ell}{n}$. Therefore the following equality holds.
\begin{equation}\label{E:anglesetcos}
\beta_{r(q+t)}  = \lambda_2\cos(\gamma_1+t\rho) + \mu_2\cos(\gamma_2-t\rho)\,.
\end{equation}
As $u\modulo{r\ell}{n}\not=\modulo{ur\ell}{n}$, we have $u\modulo{\frac{r\ell}{n}2\pi}{2\pi}\not=\modulo{u\frac{r\ell}{n}2\pi}{2\pi}$, therefore $u\modulo{\rho}{2\pi}\not=\modulo{u\rho}{2\pi}$. As $\modulo{ur\ell}{n}\le\sqrt{n}$ we have $\modulo{u\rho}{2\pi} = \modulo{u\frac{r\ell}{n}2\pi}{2\pi} = \frac{\modulo{ur\ell}{n}}{n} 2\pi\le \frac{2\pi}{\sqrt{n}}<\frac{\pi}{2}$. Hence by Lemma~\ref{L:somecospositive} there is $\module{t}<\module{u}$ such that
\[
\beta_{r ( q+t )}  = \lambda_2\cos(\theta_1+t\rho) + \mu_2\cos(\theta_2-t\rho)\ge 0\,.
\]
Therefore by \eqref{E:minorealpha} we have $\alpha_{r(q+t)}+\beta_{r(q+t)}\ge\sqrt{s}-\varepsilon$.
\end{proof}

With a similar and simpler proof, we obtain the following variation of Theorem~\ref{T:conjuguecascomplexe}.

\begin{theorem}\label{T:conjuguecascomplexeavecthetasur2}
Let $\lambda_1,\mu_1,\lambda_2$ be real numbers, and $\varepsilon>0$. Set 
\[
A=\max\left(1,2\Ceil{2\pi\frac{\module{\lambda_1}+\module{\mu_1}}{\varepsilon}}\right)\,.
\]
Let $\theta_1, \theta_2$ be in $\R$. Set
\[
s=\lambda_1^2+\mu_1^2+2\lambda_1\mu_1\cos(\theta_1)\,.
\]
Then for all $n\ge A$ even, and all $1\le k,\ell\le n-1$ with $\gcd(k,\frac{n}{2})=1$ there is $t\in\Z$ such that
\[
\lambda_1\cos\frac{k(\theta_1+2\pi t)}{n}+\mu_1\cos\frac{(n-k)(\theta_1+2\pi t)}{n}+\lambda_2\cos\frac{\theta_2+2\pi t}{2}>\sqrt{s}-\varepsilon\,.
\]
\end{theorem}

\section{Determination of the Julia Robinson Number}\label{S:JR}

Let $0<a<b$ be coprime integers, let $\theta=\arccos(\frac{a}{b})\in]0,\frac{\pi}{2}[$. Denote $\alpha=e^{i\theta}=\frac{a+i\sqrt{b^2-a^2}}{b}$. Let $r\ge 2$ be an integer. Assume the following:
\begin{itemize}
\item $4$ divides $b$;
\item $v_2 ( b )-1$ is coprime to $r$;
\item $r$ divides $b$;
\item For all odd prime $p$ dividing $b$, we have $v_p ( b )$ coprime to $r$.
\end{itemize}
We fix such $a,b,r$ in the whole section. Subsequentely if $(a,b,r)$ satisfies those conditions we say that $(a,b,r)$ is good.

Given $n\ge 1$, denote $L_n=\mathbb{Q}(e^{\frac{i\theta}{n}})$, and $K_n=\mathbb{Q}(\cos(\frac{\theta}{n}))$.
Moreover, set
\[
L=\bigcup_{k\in\mathbb{N}}L_{r^k}\quad\text{ and } \quad K=\bigcup_{k\in\mathbb{N}}K_{r^k}.
\]
We say that $K$ is the totally real field associated to $(a,b,r)$. Denote by $N$ the smallest common multiple of $v_2 ( b )-1$, and all $v_p ( b )$ with $p$ odd prime.

Let $u$ be the square factor of $b+a$, and $v$ be the square factor of $b-a$, that is $\frac{b+a}{u^2}$ (respectively $\frac{b-a}{v^2}$) are square-free. Note that the conditions in the following notation are similar to (2)-(4) in Lemma \ref{L:caracterisionentieralgebrique}

\begin{notation}\label{N:notJR}
Let $0\le k\le N-1$. Denote by $S_k$ the set of all pairs $(x,y)\in\Q^2$ that satisfies the following conditions:
\begin{enumerate}
\item For all odd prime number $p$ such that $p$ divides $b$ we have $v_p(x)\ge \frac{k+1}{N} v_p(b)$ and $v_p(y)\ge \frac{N-k}{N} v_p(b)$;
\item We have $v_2(x)\ge 1+ \frac{k+1}{N} (v_2(b)-1)$ and $v_2(y)\ge 1+\frac{N-k}{N}( v_2(b)-1)$;
\item $u(x+y)$ and $v(x-y)$ are integers.
\end{enumerate}
We denote by $U=S_1\cup\dots S_{N-1}$.
\end{notation}

\begin{notation}\label{N:notJR2}
As before we consider the following map
\begin{align*}
N\colon\R^2 &\to\R\\
(x,y)&\mapsto N(x,y)=x^2+y^2+2xy\frac{a}{b}\,.
\end{align*}
Set $s=\inf\Famm{\sqrt{N(x,y)}}{0\le k\le N-1\text{ and }(x,y)\in S_k}$. The goal of this section is to prove that the Julia Robinson Number of $\cO_K$ is $\Ceil{s}+s$.
\end{notation}

Note that $(x,y)\in S_k$ if and only if $(y,x)\in S_{N-k-1}$. We say that $(\theta,K,s,U,N\colon \R^2\to \R)$ are the \emph{objects associated to $(a,b,r)$}.

\begin{lemma}\label{lemJR1}
Let $1\le k<N$. Then $S_k$ is a discrete $\Z$-submodule of $\Q^2$.
\end{lemma}

\begin{proof}
By the well-known property that for every valuation $v$ over a field $F$ and $\alpha, \beta \in F$, $v ( \alpha + \beta ) \geq \min ( v ( \alpha ) , v ( \beta ) )$, we immediately have that if $( x_1 , y_1 )$ and $( x_2 , y_2 )$ satisfy the conditions (1) and (2) of Notation~\ref{N:notJR}, then $( x_1 + x_2 , y_1 + y_2 )$ also satisfies those conditions.
Finally, it is clear that if $( x_1 , y_1 )$ and $( x_2 , y_2 )$ satisfy the conditions (3) of Notation \ref{N:notJR}, then $( x_1 + x_2 , y_1 + y_2 )$ also.
Thus $S_k$ is a subgroup of $\Q^2$.

Observe that condition (3) in Notations \ref{N:notJR} implies that if $( x, y ) \in S_k$, then $2uv x \in \Z$ and $2uv y \in \Z$.
Then $S_k$ is contained in the discrete sub-$\Z$-module generated by $( 1 / 2uv , 0 )$ and $( 0, 1/ 2uv )$.
In particular it is a discrete $\Z$-module.
\end{proof}

\begin{lemma}\label{L:redefJR}
Let $0\le k\le N-1$. Let $x,y\in\Q$.Then for all $q\in]\frac{k}{N},\frac{k+1}{N}[\cap\Q$, we have $(x,y)\in S_k$ if and only if the following three conditions hold:
\begin{enumerate}
\item For all odd prime number $p$ such that $p$ divides $b$ we have $v_p(x)\ge q v_p(b)$ and $v_p(y)\ge (1-q) v_p(b)$;
\item We have $v_2(x)\ge 1+q (v_2(b)-1)$ and $v_2(y)\ge 1+(1-q)( v_2(b)-1)$;
\item $u(x+y)$ and $v(x-y)$ are integers.
\end{enumerate}
\end{lemma}

\begin{proof}
Let $q\in]\frac{k}{N},\frac{k+1}{N}[\cap\Q$. Let $p$ be an odd prime number dividing $b$ and let $(x,y)$ be in $S_k$.

Assume that there is an integer $\ell$ such that $\frac{k v_p(b)}{N}<\ell < \frac{(k+1) v_p(b)}{N}$, then $k<\frac{N}{v_p(b)}\ell <k+1$. However, $v_p(b)|N$, hence $\frac{N}{v_p(b)}\ell$ is an integer; a contradiction. Therefore there is no integer in $]\frac{k v_p(b)}{N},\frac{(k+1) v_p(b)}{N}[$. Note that $q v_p(b)\in]\frac{k v_p(b)}{N},\frac{(k+1) v_p(b)}{N}[$. As $v_p(x)$ is an integer it follows that $v_p(x)\ge q v_p(b)$ if and only if $v_p(x)\ge \frac{k+1}{N} v_p(b)$. Similarly $v_p(y)\ge (1-q) v_p(b)$ if and only if $v_p(y)\ge \frac{N-k}{N} v_p(b)$. Therefore the condition Lemma~\ref{L:redefJR}(1) and Notation~\ref{N:notJR}(1) are equivalent.

Similarly Lemma~\ref{L:redefJR}(2) and Notation~\ref{N:notJR}(2) are equivalent. The last condition is the same.
\end{proof}

\begin{lemma}\label{L:descriptionentieralgebriqueavecSk}
Let $n\ge 1$ be an integer dividing a power of $r$. Let $\lambda_0,\lambda_1,\dots,\lambda_{n-1}$ be in $\Q$. The following statement are equivalent:
\begin{enumerate}
\item $\lambda_0 + \sum_{\ell=1}^{n-1}\lambda_\ell\cos(\frac{\ell\theta}{n})$ is in $\cO_K$.
\item $\lambda_0\in\Z$, and for all $1\le \ell\le n-1$ we have $(\lambda_\ell,\lambda_{n-\ell})\in S_k$, where $0\le k\le N-1$ such that $\frac{\ell}{n}\in]\frac{k}{N},\frac{k+1}{N}[$.
\end{enumerate}
\end{lemma}

\begin{proof}
Assume $(1)$. Observe that $N$ is the product of $v_2(b)-1$ and all $v_p(b)$ for $p|b$ odd primes, which are each coprime to $r$, thus $N$ and $r$ are coprime. As $n$ divides a power of $r$ it follows that $n$ and $N$ are coprime, therefore for all $1\le\ell\le n-1$ and all $0\le k\le N$ we have $\frac{\ell}{n}\not=\frac{k}{N}$.

By Lemma~\ref{L:caracterisionentieralgebrique} we have $\lambda_0\in\Z$, and for all $1\le\ell\le n-1$ the following statement holds.
\begin{enumerate}
\item[$(i)$] For all odd prime number $p$ dividing $b$, $v_p ( \lambda_{\ell} ) \geq \Ceil{\frac{ \ell v_p ( b )}{n}}$;
\item[$(ii)$]$v_2 ( \lambda_{\ell} ) \geq 1 + \Ceil{\frac{\ell ( v_2 ( b ) - 1 )}{n}}$;
\item[$(iii)$] $u ( \lambda_{\ell} + \lambda_{n-\ell} )$ and $v ( \lambda_{\ell} + \lambda_{n-\ell} )$ are integers.
\end{enumerate}
Let $1\le \ell\le n-1$. Let $0\le k\le N-1$ be such that $\frac{\ell}{n}\in[\frac{k}{N}, \frac{k+1}{N}[$. As $\frac{\ell}{n}\not=\frac{k}{N}$, we have $\frac{\ell}{n}\in]\frac{k}{N}, \frac{k+1}{N}[$. Hence it follows from Lemma~\ref{L:redefJR} that $(\lambda_{\ell},\lambda_{n-\ell})\in S_k$. Therefore $(2)$ holds.

Assume $(2)$. Let $1\le\ell\le n-1$. Let $0\le k\le N-1$ such that $\frac{\ell}{n}\in]\frac{k}{N}, \frac{k+1}{N}[$. As $(\lambda_{\ell},\lambda_{n-\ell})\in S_k$ it follows from Lemma~\ref{L:redefJR} that $(i)$, $(ii)$, and $(iii)$ holds. By Lemma~\ref{L:caracterisionentieralgebrique}, we have that $\lambda_0 + \sum_{\ell=1}^{n-1}\lambda_\ell\cos(\frac{\ell\theta}{n})$ is an algebraic integer, moreover it belongs to $K$, therefore $(1)$ holds.
\end{proof}

\begin{lemma}\label{lemJR2}
Let $M$ be a discrete $\Z$-submodule of $\Q^2$. Then
\begin{enumerate}
\item For all $C\in \R$ there are only finitely many $(x,y)\in M$ such that $N(x,y)\le C$.
\item If $M$ is not trivial, then there is $(x,y)\in M\setminus\set{(0,0)}$ such that, for all $(x',y')\in M\setminus\set{(0,0)}$ we have $N(x,y)\le N(x',y')$.
\end{enumerate}
\end{lemma}

\begin{proof}
Let $C$ be a real number and consider the affine conic with equation $X^2 + Y^2 + 2XY\frac{a}{b} = C$.
Since $0 < a < b$ by hypothesis, the locus of its zeros in $\R^2$ is an ellipse. 
Then the set $E_C = \{ ( x, y ) \in \R^2 \ N ( x, y ) \leq C \}$ is compact. 
Since $M$ is discrete, there are only finitely many $(x,y)\in M$ such that $N(x,y)\le C$.

By hypothesis $M$ is discrete, then every point of $M$ is isolated.
Then take $C_1 \in \R$ such that $E_{C_1}$ only contains $( 0, 0 )$ and $C_2 \in \R$ such that $E_{C_2}$ contains at least a point of $M$ distinct from $( 0, 0 )$ (observe that $C_2$ exists because $M$ is not trivial).
Then the set $B$ given by $E_{C_2}$ minus the interior part of $E_{C_1}$ is a compact, and so it contains a finite number of elements of $M$.
Take $( x, y ) \in B$ such that $N ( x, y ) \leq N ( w , z )$ for every $( w, z ) \in B$.
Then $( x, y ) \neq ( 0, 0 )$ and $N ( x, y ) \leq N ( x' , y' )$ for every $( x' , y' ) \in M \setminus\set{(0,0)}$, because if $( x' , y' ) \not \in B$, then $N ( x', y' ) \geq C_2 \geq N ( x, y )$.
\end{proof}

\begin{lemma}\label{lemJR3}
Let $(\lambda,\mu)$ be in $S_k$. Then there are infinitely many $q\in \Q\cap[0,1]$ such that $\lambda\cos(q\theta)+\mu\cos( (1-q)\theta)\in\cO_K$. Moreover we can choose $q$ such that the denominator of $q$ is any large enough divisor of a power of $r$.
\end{lemma}

\begin{proof}
Take coprime integers $t$ and $n=r^e$ such that $\module{\frac{2k + 1}{2N} - \frac{t}{n}} < \frac{1}{2N}$, Note that there are infinitely many $t,n$ that satisfy this condition. In particular $\frac{t}{n}$ is in the interval $]\frac{k}{N}, \frac{k+1}{N}[$. It follows that $\frac{n-t}{n}\in ]\frac{N - k - 1}{N}, \frac{N-k}{N}[$, moreover $(\mu,\lambda)\in S_{N-k-1}$. Therefore by Lemma~\ref{L:descriptionentieralgebriqueavecSk} we have $\lambda\cos(\frac{t}{n}\theta)+\mu\cos( (\frac{n-t}{n})\theta)\in\cO_K$.
\end{proof} 

\begin{corollary}\label{C:preencadrementJR}
Denote by $t$ the product of all odd primes (without counting multiplicity) dividing $b$. The following statements hold.
\begin{enumerate}
\item $s^2$ is an integer, and $8t|s^2$.
\item $s\ge 2\sqrt{2}$.
\item $s\le 4t$.
\item $s^2\le 8t\frac{b-a}{v^2}$ and $s^2\le 8t\frac{b+a}{u^2}$. In particular if $b-a$ is a square then $s^2=8t$.
\item If $(b-a)$ and $(b+a)$ are square-free and $b\ge 2t+a$, then $s=4t$.
\end{enumerate}
\end{corollary}

\begin{proof}
It follows from Lemma~\ref{lemJR1} and Lemma~\ref{lemJR2} that there is $0\le k\le N-1$ and $(\lambda,\mu)\in S_k\setminus\set{(0,0)}$ such that $s=\sqrt{N(\lambda,\mu)}$. So $s^2=\lambda^2+\mu^2+2\lambda\mu\frac{a}{b}$. 

Let $q\not=\frac{1}{2}$ be as in Lemma~\ref{lemJR3}, that is $\lambda\cos(q\theta)+\mu\cos( (1-q)\theta)\in\cO_K$. So $T( (\lambda\cos(q\theta)+\mu\cos( (1-q)\theta))^2)=\frac{1}{2}N(\lambda,\mu)=\frac{1}{2}s^2$ is an integer.

Let $p$ be an odd prime number dividing $b$. It follows from the definition of $S_k$ that
\begin{equation}\label{E:pvaluationlambdaetmu}
v_p(\lambda)\ge\Ceil{(k+1)\frac{v_p(b)}{N}}\ge 1\quad\text{ and }\quad v_p(\mu)\ge\Ceil{(N-k)\frac{v_p(b)}{N}}\ge 1\,.
\end{equation}
Therefore the following statement holds
\begin{align}
v_p(\lambda\mu)&=v_p(\lambda)+v_p(\mu)\\
&\ge  \Ceil{(k+1)\frac{v_p(b)}{N}} + \Ceil{(N-k)\frac{v_p(b)}{N}} \\
&\ge \Ceil{(N+1)\frac{v_p(b)}{N}}\\
&\ge v_p(b)+1\,.\label{E:pvaluationlambdaxmu}
\end{align}
Hence $v_p(2\lambda\mu\frac{a}{b})\ge 1$, therefore
\begin{equation}\label{E:pvaluations2}
v_p(s^2)=v_p(\lambda^2+\mu^2+2\lambda\mu\frac{a}{b})\ge\min(v_p(\lambda^2),v_p(\mu^2),v_p(2\lambda\mu\frac{a}{b}))\ge 1=v_p(8t)\,.
\end{equation}

Similarly the following statement hold
\begin{equation}\label{E:2valuationlambdaetmu}
v_2(\lambda) \ge \Ceil{1 + (k + 1)\frac{v_2(b) - 1}{N}}\ge 2\quad\text{ and }\quad v_2(\mu)\ge\Ceil{1 + (N-k)\frac{v_2(b) - 1}{N}}\ge 2\,.
\end{equation}
Therefore
\begin{align}
v_2(\lambda\mu) &\ge  \Ceil{1+(k+1)\frac{v_2(b)-1}{N}} + \Ceil{1+(N-k)\frac{v_2(b)-1}{N}}\\
& \ge \Ceil{2+(N+1)\frac{v_2(b)-1}{N}}\\
& \ge 2 + v_2(b)\,.\label{E:2valuationlambdaxmu}
\end{align}
So $v_2(2\lambda\mu\frac{a}{b})\ge 3$, therefore 
\begin{equation}\label{E:2valuations2}
v_2(s^2) = v_2(\lambda^2+\mu^2+2\lambda\mu\frac{a}{b})\ge 3=v_2(8t)\,.
\end{equation}
It follows from \eqref{E:pvaluations2} and \eqref{E:2valuations2} that $8t|s^2$, so $(1)$ holds. In particular $8\le s^2$, so $s\ge 2\sqrt{2}$, that is $(2)$ holds.

Let $p$ be an odd prime number dividing $b$. Note that $v_p(4t)=1=\Ceil{\frac{v_p(b)}{N}}$ and $v_2(4t)=2=\Ceil{1+\frac{v_2(b)-1}{N}}$, and $4tu$ and $4tv$ are integers (as $4t$ is an integer) thus $(4t,0)\in S_0$. Therefore $s\le \sqrt{N(4t,0)}=4t$, so $(3)$ holds.

Set $k=\frac{N-1}{2}$. Set $x=2\radi{\frac{b}{2}}{\frac{1}{2}}$. Note that $v_2(x)= 1+\ceil{\frac{v_2(b)-1}{2}}$. As $v_2(b)-1$ is relatively prime to $b$, it follows that $v_2(b)-1$ is odd, thus $v_2(x)=1+\frac{v_2(b)}{2}$. Note that $N\ge v_2(b)-1$ thus the following inequality holds
\begin{equation}\label{E:valuationcorencadrement1}
\frac{N+1}{2N}(v_2(b)-1)+1 = \frac{v_2(b)-1}{2} + \frac{v_2(b)-1}{2N} + 1 \le \frac{v_2(b)-1}{2} + \frac{1}{2} + 1 = v_2(x)\,.
\end{equation}
Let $p$ be an odd prime number dividing $b$. Then $v_p(x) = \ceil{\frac{v_p(b)}{2}} = \frac{v_p(b)+1}{2}$, as $v_p(b)$ is odd. As $N\ge v_p(b)$, the following inequality holds
\begin{equation}\label{E:valuationcorencadrement2}
\frac{N+1}{2N}v_p(b) = \frac{v_p(b)}{2} + \frac{v_p(b)}{2N} \le \frac{v_p(b)}{2} + \frac{1}{2} = v_p(x)\,.
\end{equation}

As $v$ is relatively prime to $b$ and $\frac{k+1}{N} = \frac{N-k}{N} = \frac{N+1}{2N}$, it follows from \eqref{E:valuationcorencadrement1} and \eqref{E:valuationcorencadrement2} that $(\frac{x}{v},-\frac{x}{v})$ satisfies the condition (1) and (2) of Notation~\ref{N:notJR}. Moreover $u(\frac{x}{v}-\frac{x}{v})=0$ and $v(\frac{x}{v}+\frac{x}{v})=2x$ are integers, hence $(\frac{x}{v},-\frac{x}{v})\in S_k$. Therefore the following inequality holds
\begin{equation}\label{E:corencadrementmajorers}
s \le N\left(\frac{x}{v},-\frac{x}{v}\right)= 2\frac{x^2}{v^2}\left(1-\frac{a}{b}\right)\,.
\end{equation}
However, as $v_2(x)=1+\frac{v_2(b)}{2}$ and for all odd prime $p$ dividing $b$ we have $v_p(x) = \frac{v_p(b)+1}{2}$, and for all other prime $p$ we have $v_p(x)=0$, it follows that $x^2=4tb$. It follows from \eqref{E:corencadrementmajorers} that $s\le 2 \frac{4tb}{v^2}\frac{b-a}{b} = 8t\frac{b-a}{v^2}$.

A similar proof, considering $(\frac{x}{u},\frac{x}{u})\in S_k$, yelds $s\le 8t\frac{b+a}{u^2}$. Hence $(4)$ holds.

If $b\ge 2t+a$, $b-a$ and $b+a$ are square-free then $u=v=1$. Thus $\lambda+\mu$ and $\lambda-\mu$ are integers. However as seen in \eqref{E:2valuationlambdaetmu}, we have $v_2(\lambda)\ge 1$, and $v_2(\mu)\ge 1$. Therefore $\lambda$ and $\mu$ are integers. Moreover it follows from \eqref{E:pvaluationlambdaetmu} and \eqref{E:2valuationlambdaetmu} that $4t$ divides $\lambda$ and $\mu$.

If $\lambda\mu\ge 0$ then $s^2=N(\lambda,\mu) = \lambda^2+\mu^2+2\lambda\mu\frac{a}{b}\ge \max(\lambda^2,\mu^2)$ however $4t$ divides both $\lambda$ and $\mu$, and $(\lambda,\mu)\not=(0,0)$, thus $s\ge 4t$.

Assume that $\lambda\mu<0$. The following statement holds
\begin{equation*}
N(\lambda,\mu) = \lambda^2+\mu^2+2\lambda\mu\frac{a}{b} = (\lambda+\mu)^2 - 2\lambda\mu\left(1-\frac{a}{b}\right)\ge (\lambda+\mu)^2\,.
\end{equation*}
However $4t|\lambda+\mu$, so if $\lambda+\mu\not=0$ we have $s=\sqrt{N(\lambda,\mu)}\ge 4t$.

Now assume that $\lambda+\mu=0$, that is $\lambda=-\mu$. The following holds
\[
N(\lambda,\mu) = \lambda^2+\lambda^2-2\lambda^2\frac{a}{b}  = 2\lambda^2\left(1-\frac{a}{b}\right)\,.
\]
It follows from \eqref{E:pvaluationlambdaxmu} that $v_p(\lambda^2)=v_p(\lambda\mu)\ge v_p(b)+1=v_p(4bt)$, similarly \eqref{E:2valuationlambdaxmu} implies $v_2(\lambda^2)=v_2(\lambda\mu)=2+v_2(b)=v_2(4tb)$. Therefore $4tb$ divides $\lambda^2$, so $N(\lambda,\mu)\ge 8tb\left(1-\frac{a}{b}\right)=8t(b-a)\ge 16t^2=(4t)^2$, and so $s=\sqrt{N(\lambda,\mu)}\ge 4t$.

However, from $(2)$ we also have $s\le 4t$, thus $s=4t$. Therefore $(5)$ holds.
\end{proof}

\begin{corollary}\label{C:corJR1}
There are infinitely many $\beta\in\cO_K$ such that all conjugates of $\beta$ are in $[0,\Ceil{s}+s]$. In particular the Julia Robinson Number of $\cO_K$ is at most $\Ceil{s}+s$.
\end{corollary}

\begin{proof}
By Lemma \ref{lemJR1}, for every $k$ between $0$ and $N-1$, $S_k$ is a discrete $Z$-module. 
Then, by Lemma \ref{lemJR2}, there exists $k$ and $( \lambda , \mu ) \in S_k$ such that $N ( \lambda , \mu ) = s$.
By Lemma \ref{lemJR3}, there exist infinitely many  $q \in [0, 1] \cap \Q$ such that $\gamma_q = \lambda\cos(q\theta)+\mu\cos( (1-q)\theta) \in \cO_K$. From Lemma~\ref{L:enumerationconjugue} we see that the conjugates of $\gamma_q$ are the $\lambda\cos(q\theta+2\pi qt)+\mu\cos( (1-q)\theta-2\pi qt)$, with $t\in\Z$. Hence by Lemma~\ref{L:maxsommecos}, every conjugate $\gamma_q'$ of $\gamma$ satisfies
\[
\module{\gamma_q'} \le \sqrt{\lambda^2+\mu^2+2\lambda\mu\cos(\theta)}= \sqrt{N ( \lambda , \mu )} = s.
\]
Hence every conjugate of $\Ceil{s}+\gamma_q$ is positive, and it is smaller than $\Ceil{s}+s$.
This proves that the Julia Robinson Number of $\cO_K$ is at most $\Ceil{s}+s$. 
\end{proof}

\begin{theorem}\label{T:MainJuliaRobinsonNumber}
The Julia Robinson Number of $\cO_K$ is $\Ceil{s}+s$. Moreover $\cO_K$ has the Julia Robinson Property.
\end{theorem}

\begin{proof}
From Corollary~\ref{C:corJR1} we already know that the Julia Robinson Number of $\cO_K$ is at most $\Ceil{s}+s$. Moreover there are infinitely many $\beta\in\cO_K$ such that all conjugates of $\beta$ are in $[0,\Ceil{s}+s]$, thus to prove that $\cO_K$ has the Julia Robinson Property, we only have to prove that the Julia Robinson Number of $\cO_K$ is $ \Ceil{s}+s$.

Let $\varepsilon>0$ be a real number. We can assume that $\varepsilon\le 1$, and if $s$ is not an integer we also assume that $\varepsilon$ is smaller than the fractional part of $s$. In particular
\begin{equation}\label{E:majorfractional}
q-s+\varepsilon < 0\,,\text{for all integer $q< s$.}
\end{equation}

Denote by $S$ the set of all pairs $(\lambda,\mu)\in\bigcup_{k=0}^{N-1}S_k$ such that $N(\lambda,\mu)\le 2s^2+s+1$. We see from Lemma~\ref{lemJR2} that $S$ is finite. Let $B$ be a power of $r$ such that $B\ge 17$ and
\[
B\ge \Ceil{4\pi\frac{\module{\lambda}+\module{\mu}}{\varepsilon}}^2\,,\quad\text{for all $(\lambda,\mu)\in S$.}
\]
Set 
\[
X=\Setm{\lambda\cos\frac{t\theta}{d} + \mu\cos\frac{(d-t)\theta}{d}}{d\le B,\ 1\le t\le d-1\text{, and }(\lambda,\mu)\in S}
\]
and
\[
X'=\Setm{\lambda\cos\frac{\theta}{2} }{(\lambda,\lambda)\in S}\,.
\]
In particular $0\in X$, and so $X\subseteq X+X$. Set $Y=\set{1,2,\dots,2\ceil{s}-1} + X + (X\cup X')$.

Let $\gamma\in\cO_K$ be such that all conjugates of $\gamma$ are in $[0,\ceil{s}+s]$. There is $n\in\N$ such that $\beta\in K_n=\Q(\cos(\frac{\theta}{n}))$. By Lemma~\ref{L:basecosinus} we can write $\gamma=\lambda_0+\sum_{t=1}^{n-1}\lambda_t\cos\frac{t\theta}{n}$, with $\lambda_0,\lambda_1,\dots,\lambda_{n-1}\in\Q$. From Lemma~\ref{L:caracterisionentieralgebrique} we see that $\lambda_0\in\Z$, and for all $1\le t\le n-1$ we have $(\lambda_t,\lambda_{n-t})$ in some $S_k$, with $0\le k\le N-1$. So, by construction of $s$, either $(\lambda_t,\lambda_{n-t})=(0,0)$ or $N(\lambda_t,\lambda_{n-t})\ge s^2$. Set $\beta=\gamma-\lambda_0=\sum_{t=1}^{n-1}\lambda_k\cos\frac{t\theta}{n}$. It follows from Lemma~\ref{L:basecosinus} that $T(\beta)=0$.

As $\gamma=\lambda_0+\beta$ has no negative conjugates, it follows from Lemma~\ref{L:trace0} that $\lambda_0>0$, similarly we have $\lambda_0<2\ceil{s}$. Moreover the maximal difference of two conjugates of $\beta$ is $\delta(\beta)=\delta(\gamma)\le \ceil{s}+s$. The trace of $\beta^2$ is computed in Lemma~\ref{L:TSommeCos}, that is $T(\beta^2) = \frac{1}{2}\sum_{k=1}^{\floor{\frac{n}{2}}} N(\lambda_k,\lambda_{n-k})$. Moreover from Lemma~\ref{L:trace3} we have $\delta(\beta)\ge 2\sqrt{T(\beta^2)}$. Therefore we obtain the following inequality
\[
2\sqrt{\frac{1}{2}\sum_{k=1}^{\floor{\frac{n}{2}}} N(\lambda_k,\lambda_{n-k})}\le \ceil{s}+s \le 2s+1\,.
\]
Thus
\begin{equation}\label{E:major1}
2\sum_{k=1}^{\floor{\frac{n}{2}}} N(\lambda_k,\lambda_{n-k})\le (2s+1)^2\,.
\end{equation}
Let $1\le k\le n-1$. If $N(\lambda_k,\lambda_{n-k})>2s^2+s+1$ then from \eqref{E:major1} we have
\[
4s^2+2s+2=2(2s^2+s+1)<2N(\lambda,\mu)\le (2s+1)^2=4s^2+2s+1\,;
\]
a contradiction. Therefore $N(\lambda_k,\lambda_{n-k})\le 2s^2+s+1$. Hence we have $(\lambda_k,\lambda_{n-k})\in S$, for all $1\le k\le n-1$.

Let $C$ be the number of $k\in\set{1,\dots,\floor{\frac{n}{2}}}$ such that $(\lambda_k,\lambda_{n-k})\not=(0,0)$. For each such $k$ we have $N(\lambda_k,\lambda_{n-k})\ge s^2$. Thus by \eqref{E:major1} we have $2Cs^2\le (2s+1)^2$, so $C\le 2+\frac{1}{s} + \frac{1}{2s^2}$. However, from Lemma~\ref{C:preencadrementJR} we have $s\ge 2\sqrt{2}$. As $C$ is an integer it follows that $C\le 2$.

If $C=0$, then $\beta=0$, and so $\gamma=\lambda\in Y$. If $C=1$ then we can write $\beta = \lambda \cos\frac{\theta}{2}$ with $(\lambda,\lambda)\in S$, or $\beta = \lambda \cos\frac{k\theta}{n} + \mu \cos\frac{k\theta}{n}$, with $1\le k\le n-1$, and $(\lambda,\mu)\in S$. In the first case we see that $\gamma=\lambda_0+\beta\in Y$. We now treat the second case. We can assume that $\gcd(k,n)=1$.

Assume that $n\ge B$. Thus $n\ge A$, so from Corollary~\ref{C:grandconjuguecassimple} and Lemma~\ref{L:enumerationconjugue} we obtain a conjugate $\beta'$ of $\beta$ such that $\beta'\ge\sqrt{\lambda^2+\mu^2+2\lambda\mu\cos\theta} -\varepsilon = N(\lambda,\mu)-\varepsilon\ge s-\varepsilon$. Similarly we have a conjugate $\beta''$ of $\beta$ such that $\beta''\le -s+\varepsilon$.

If $\lambda_0\ge \ceil{s}$, then a conjugate of $\gamma$ is $\lambda_0+\beta'\ge\ceil{s}+ s-\varepsilon$; a contradiction. If $\lambda_0<\ceil{s}$, then $\lambda_0<s$, so a conjugate of $\gamma$ is $\lambda_0+\gamma''\le \lambda_0-s+\varepsilon$ which is negative by \eqref{E:majorfractional}; a contradiction. Therefore $n\le B$, and so $\gamma\in Y$.

If $C=2$, once again simplifying fractions, then we can write $\beta = \lambda_1 \cos\frac{k\theta}{n} + \mu_1 \cos\frac{(n-k)\theta}{n} + \lambda_2 \cos\frac{\theta}{2}$ with $(\lambda_2,\lambda_2)\in S$, $(\lambda_1,\mu_1)\in S$, and $\gcd(k,\frac{n}{2})=1$, or $\beta = \lambda_1 \cos\frac{k\theta}{n} + \mu_1 \cos\frac{(n-k)\theta}{n} + \lambda_2 \cos\frac{\ell\theta}{n} + \mu_2 \cos\frac{(n-\ell)\theta}{n}$, with $1\le k<\ell<n-\ell\le n-1$, $(\lambda_1,\mu_1)\in S$, $(\lambda_2,\mu_2)\in S$, and $\gcd(k,\ell,n)=1$. In both case, assuming $n\ge B$, using respectively Theorem~\ref{T:conjuguecascomplexeavecthetasur2} and Theorem~\ref{T:conjuguecascomplexe} we have conjugates $\beta'$ and $\beta''$ of $\beta$ such that $\beta'\ge s-\varepsilon$, and $\beta''\le -s+\varepsilon$. As before we reach a contradiction. Therefore $n\le B$ and so $\gamma=\lambda_0+\beta\in Y$.
\end{proof}

As $\cO_K$ has the Julia Robinson Property, from the result of Robinson in  \cite{Robinson} we obtain the following Corollary of Theorem~\ref{T:MainJuliaRobinsonNumber}.

\begin{corollary}\label{C:indec}
The first order theory of $\cO_K$ is undecidable.
\end{corollary}

The following lemma express that there are not many ``simple'' square roots of rationals in $K$. It will be used to prove that many of the constructed fields are distincts.

\begin{lemma}\label{L:peuderacinescarre}
Let $(\lambda,\mu)\in\Q\setminus\set{(0,0)}$, let $n\ge 2$ be an integer dividing a power of $r$, let $1\le k\le n-1$. Then $\left(\lambda\cos\frac{k\theta}{n} + \mu\cos\frac{(n-k)\theta}{n}\right)^2\in\Q$ if and only if $2k=n$.
\end{lemma}

\begin{proof}
Set $\beta=\lambda\cos\frac{k\theta}{n} + \mu\cos\frac{(n-k)\theta}{n}$. Note that $\left(\cos\frac{\theta}{2}\right)^2=\frac{1+\cos\theta}{2} = \frac{b+a}{2b}$ belongs to $\Q$. Thus if $n=2k$, then the following statement holds
\[
\beta^2 = \left((\lambda+\mu)\cos\frac{\theta}{2}\right)^2\in\Q\,.
\]

Assume that $n\not=2k$, hence $k\not=n-k$. We can assume that $k<n-k$. It follows from Lemma~\ref{L:Euler2} that
\begin{equation}\label{E:betacarre}
\beta^2= \frac{1}{2}\left(\lambda^2+\mu^2+2\lambda\mu\frac{a}{b} 
 + (\lambda^2-\mu^2) \cos\frac{2k\theta}{n} + 2(\lambda\mu+\mu^2\frac{a}{b}) \cos\frac{(n-2k)\theta}{n}\right).
\end{equation}
Note that if $ (\lambda^2-\mu^2) =0$ then $\lambda=\pm\mu\not=0$, hence $(\lambda\mu+\mu^2\frac{a}{b}) =\mu^2(\frac{a}{b}\pm 1)\not=0$. Assume that $2k\not=n-2k$, it follows from Lemma~\ref{L:basecosinus} that $1,  \cos\frac{2k\theta}{n}, \cos\frac{n-2k\theta}{n}$ are $\Q$-linearly independent, thus $\beta^2\not\in\Q$.

Assume that $2k=n-2k$, thus $4k=n$, so $n$ is even, and so $r$ is even. Moreover from \eqref{E:betacarre} we have
\[
\beta^2= \frac{1}{2}\left(\lambda^2+\mu^2+2\lambda\mu\frac{a}{b} + (\lambda^2-\mu^2 + 2\lambda\mu+2\mu^2\frac{a}{b})  \cos\frac{\theta}{2}\right)\,.
\]
Assume that $\beta^2\in\Q$. However from Lemma~\ref{L:basecosinus} we see that $1,\cos\frac{\theta}{2}$ are linearly independent, thus $\lambda^2-\mu^2 + 2\lambda\mu+2\mu^2\frac{a}{b}=0$. Hence $\lambda$ is a root of the polynomial $\Lambda^2+2\mu\Lambda + 2\mu^2\frac{a}{b} - \mu^2$. Therefore the discriminant $\Delta=4\mu^2-4(2\mu^2\frac{a}{b} - \mu^2) = 8\mu^2-8\mu^2\frac{a}{b}=8\mu^2\frac{a+b}{b}$ is a square. Therefore $2\frac{a+b}{b}$ is a square, thus $v_2(2\frac{a+b}{b})=1-v_2(b)$ is even. However $r$ and $v_2(b)-1$ are relatively prime, and $r$ is even; a contradiction.
\end{proof}

\section{Example of Julia Robinson's Number}\label{sec7}

In this section we construct particular $a$ and $b$ such that we can explicitly compute the number $s$ (see beginning of Section~\ref{S:JR} and Notation~\ref{N:notJR2}). We obtain the examples of Julia Robinson Numbers in Theorem~\ref{T:exemplenombreJR}. We basically consider two family of cases. For one we choose $a,b$ such that $b-a=1$ (which is a square). The other family is such that $b-a$, $b+a$ are large and square-free, the existence of such pairs $(a,b)$ relies on the fact that the density of square-free integers is large.

\begin{notation}
Given $x\ge 0$, $k\ge 1$, and $i$ relatively prime to $k$, denote by $Q(x;i,k)$ the number of integer $j\ge 0$, such that $j\le x$ and $j\equiv i\mod k$.
\end{notation}

The following theorem was proved by Prachar \cite{Prachar} (see also Hooley \cite{Hooley}).

\begin{theorem}\label{T:countsquarefree}
Let $k\ge 1$ be an integer and $i$ relatively prime to $k$. Then we have the following equivalence
\[
Q(x;i,k)\sim_{x\to\infty} x\frac{6}{\pi^2 k} \prod_{p|k} \frac{p^2}{p^2-1}.
\]
\end{theorem}

\begin{lemma}\label{L:squarfreenumbers}
Let $k\ge 1$ be an integer, let $1>\varepsilon>0$ be a real number. Then there is an integer $C$ such that for all $b$ multiple of $k$ if $b\ge C$ then there is $1\le a\le \varepsilon b$ such that $a\equiv 1\mod k$, and $b-a$ and $b+a$ are square-free.
\end{lemma}

\begin{proof}
Set $l=\frac{6}{\pi^2 k} \prod_{p|k} \frac{p^2}{p^2-1}$. Note that $kl\ge\frac{3}{5}$. From Theorem~\ref{T:countsquarefree} we have
\[
\lim_{x\to\infty} \frac{Q(x;1,k)}{x} = l\quad\text{ and } \lim_{x\to\infty} \frac{Q(x;-1,k)}{x} = l\,.
\]
Set $3\eta = \varepsilon(l-\frac{3}{5k})$. Pick $A\ge 0$ such that for all $x\ge A$ and $i=\pm 1$ we have $\module{\frac{Q(x;i,k)}{x}-l}\le\eta$.

Set $C=\frac{A+1}{1-\varepsilon}$. Let $b\ge C$, thus $b-\varepsilon b-1\ge A$. Therefore we have
\[
\module{\frac{Q(b-\varepsilon b;i,k)}{b-\varepsilon b} -l} \le\eta\,.
\]
It follows that
\begin{equation}\label{E:bornageQ1}
\module{Q(b-\varepsilon b;i,k) -(b-\varepsilon b)l} \le (b-\varepsilon b)\eta\,.
\end{equation}
Similarly
\begin{equation}\label{E:bornageQ2}
\module{Q(b;i,k) -bl} \le b\eta
\end{equation}
and
\begin{equation}\label{E:bornageQ3}
\module{Q(b+\varepsilon b;i,k) -(b+\varepsilon b)l} \le (b+\varepsilon b)\eta\,.
\end{equation}
From \eqref{E:bornageQ1} and \eqref{E:bornageQ2} we have
\[
\module{ Q(b;i,k) - Q(b-\varepsilon b;i,k) +\varepsilon bl } \le 2b-\varepsilon b\eta \le 3b\eta \,.
\]

\begin{equation}\label{E:evaluationQ1}
Q(b;i,k) - Q(b-\varepsilon b;i,k) \ge \varepsilon bl -  3b\eta = \varepsilon bl  - \varepsilon b (l-\frac{3}{5k}) = \varepsilon b\frac{3}{5k}\,.
\end{equation}

Similarly from \eqref{E:bornageQ2} and \eqref{E:bornageQ3} we deduce
\begin{equation}\label{E:evaluationQ2}
 Q(b+\varepsilon b;i,k)  - Q(b;i,k) \ge \varepsilon bl -  3b\eta = \varepsilon b\frac{3}{5k}\,.
\end{equation}

Consider the set $X_i=\setm{b+ijk+i}{j\text{ integer and }0\le j\le \frac{\varepsilon b-1}{k}}$. Hence $X_1\subseteq [b,b+\varepsilon b]$, moreover $X_1$ is the set of all integers in $[b,b+\varepsilon b]$ congruent to 1 modulo $k$. Similarly $X_{-1}\subseteq [b-\varepsilon b,b]$, moreover $X_{-1}$ is the set of all integers in $[b-\varepsilon b,b]$ congruent to $-1$ modulo $k$.

From \eqref{E:evaluationQ1} we have, in the interval $[b-\varepsilon b,b]$, at least $Q(b;-1,k) - Q(b-\varepsilon b;-1,k)\ge \varepsilon b\frac{3}{5k}$ square-free integers congruent to $-1$ modulo $k$. However all integer in $[b-\varepsilon b,b]$ congruent to $-1$ modulo $k$ are in $X_{-1}$. Therefore there are at least $\Ceil{\varepsilon b\frac{3}{5k}}$ square-free integers in $X_{-1}$. However $X_{-1}$ has $\Floor{\frac{\varepsilon b-1}{k}}$ elements, therefore the proportion of square-free elements in $X_{-1}$ is at least $\varepsilon b\frac{3}{5k} \frac{k}{\varepsilon b} = \frac{3}{5}$. Similarly, by \eqref{E:evaluationQ2}, the proportion of square-free elements in $X_{1}$ is at least $\frac{3}{5}$.

Note that $X_1\to X_{-1}$, $r\mapsto 2b-r$ is a bijection. Therefore there is a square-free $r\in X_1$ such that $2b-r$ is square-free. Set $a=r-b$, thus $b+a=r$ is square-free, and $b-a=2b-r$ is square-free. Moreover as $r\in X_1$, we have $0<r-b \le \varepsilon b$. Also note that $a\equiv r\equiv 1\mod k$.
\end{proof}

We shall give some explicite examples of finite Julia Robinson's Number distinct from $4$. On the other hand we also want to show, for every example of Julia Robinson's Number $j$ obtained, that there exist infinitely many rings of the algebraic integers of totally real fields having Julia Robinson's Number $j$. We will use the following lemma.

\begin{lemma}\label{L:corpsdifferents}
Let $(a_1,b_1,r_1)$ be good; let $(a_2,b_2,r_2)$ be good. Denote by $K_1$ (resp., $K_2$) the totally real field associated to $(a_1,b_1,r_1)$ (resp., $(a_2,b_2,r_2)$). Assume that $r_1,r_2$ are even. If $(b_1+a_1)(b_2+a_2)b_1b_2$ is not a square then $K_1\not=K_2$.
\end{lemma}

\begin{proof}
Let $(\theta_i,K_i,U_i,s_i,N_i)$ be the objects associated to $(a_i,b_i,r_i)$, for $i=1,2$. Assume that $K_1=K_2$, then $\cO_{K_1}=\cO_{K_2}$, thus $\JR(\cO_{K_1}) = \JR(\cO_{K_2})$, so by Lemma~\ref{T:MainJuliaRobinsonNumber} we have $\ceil{s_1}+s_1=\ceil{s_2}+s_2$, and so $s_1=s_2$.

\begin{claim}
There are $n_1\ge 2$ a power of two, $n_2\ge 2$ an integer, $1\le k_1<\frac{n_1}{2}$, $1\le k_2\le n_2-1$ integers, $\lambda_1,\mu_1,\lambda_2,\mu_2$ in $\Q$ such that  
\begin{equation}\label{E:claim1}
\lambda_1\cos\frac{k_1\theta_1}{n_1} + \mu_1\cos\frac{(n_1-k_1)\theta_1}{n_1}=\lambda_2\cos\frac{k_2\theta_2}{n_2}+\mu_2\cos\frac{(n_2-k_2)\theta_2}{n_2}\not=0\,.
\end{equation}
\end{claim}

\begin{cproof}
As the infinum in the definition of $s_1$ (cf Notation~\ref{N:notJR2}) is a minimum (by Lemma~\ref{lemJR2}), there is $(\lambda_1,\mu_1)\in U_1$ such that $s_1^2=N_1(\lambda_1,\mu_1)$. It follows from Lemma~\ref{lemJR3} that there are an odd integer $k_1$ and a power of two $n_1\ge 4$ such that $1\le k_1<n_1$ and $\beta=\lambda_1\cos\frac{k_1\theta_1}{n_1} + \mu_1\cos\frac{(n_1-k_1)\theta_1}{n_1}\in\cO_{K_1}$. Up to permuting $\lambda_1$ and $\mu_1$, we can assume that $k_1<\frac{n_1}{2}$.

As $\beta\in\cO_{K_1}=\cO_{K_2}$, there is $m\ge 1$, and $\eta_0,\eta_1\dots \eta_{m-1}$ in $\Q$ such that
\[
\lambda_1\cos\frac{k_1\theta_1}{n_1} + \mu_1\cos\frac{(n-k_1)\theta_1}{n_1}=\eta_0+\sum_{\ell=1}^{n_2} \eta_\ell\cos\frac{\ell\theta_2}{n_2}\,.
\]
Therefore applying $T$ to both side of the equality, and multiplying by two, we obtain the following equality by Lemma~\ref{L:TSommeCos}.
\[
N_1(\lambda_1,\mu_1) = 2\eta_0^2 + \sum_{\ell=0}^{\floor{\frac{n_2}{2}}} N_2(\eta_\ell,\eta_{n_2-\ell})
\]
However for each $1\le\ell\le \frac{n_2}{2}$ we have either $\eta_\ell=0=\eta_{n_2-\ell}$ or $N_2(\eta_\ell,\eta_{n_2-\ell})\ge s_2^2=s_1^2=N_1(\lambda_1,\mu_1)$. As $\beta\not\in\Q$, it follows that there is $1\le p\le {n_2-1}$ such that $\eta_p\not=0$. Therefore for all $0\le\ell\le n_2-1$, if $\ell\not=p$ and $\ell\not=n_2-p$ we have $\lambda_\ell=0$, thus we obtain 
\[\tag*{\qedc}
\lambda_1\cos\frac{k_1\theta_1}{n_1} + \mu_1\cos\frac{(n_1-k_1)\theta_1}{n_1}=\eta_p\cos\frac{p\theta_2}{n_2}+\eta_{n_2-p}\cos\frac{(n_2-p)\theta_2}{n_2}\,.
\]
\renewcommand{\qedc}{}
\end{cproof}

\begin{claim}
There are $x,y$ in $\Q$ such that  
\[
x\cos\frac{\theta_1}{2}=y\cos\frac{\theta_2}{2}\not=0\,.
\]
\end{claim}

\begin{cproof}
Let $n_1,n_2,k_1,k_2,\lambda_1,\mu_1,\lambda_2,\mu_2$ as in Claim~1 with $n_1$ minimal. We can assume that $k_2\le\frac{n_2}{2}$. From Lemma~\ref{L:peuderacinescarre} and \eqref{E:claim1}, we have $2k_1=n_1$ if and only if $2k_2=n_2$.
Assume that $n_1\ge 4$.
Squaring both side we obtain from Lemma~\ref{L:Euler2}
\begin{multline}\label{E:calculsommecarre1}
\frac{1}{2}(\lambda_1^2+\mu_1^2+2\lambda_1\mu_1\frac{a_1}{b_1}) + (\lambda_1^2+\mu_1^2) \cos\frac{2k_1\theta_1}{n_1}
 + 2(\lambda_1\mu_1+\mu_1^2\frac{a_1}{b_1}) \cos\frac{(n_1-2k_1)\theta_1}{n_1}\\
 = \frac{1}{2}(\lambda_2^2+\mu_2^2+2\lambda_2\mu_2\frac{a_2}{b_2})
 + (\lambda_2^2+\mu_2^2) \cos\frac{2k_2\theta_2}{n_2}
 + 2(\lambda_2\mu_2+\mu_2^2\frac{a_2}{b_2}) \cos\frac{(n_2-2k_2)\theta_2}{n_2}
\end{multline}
From Lemma~\ref{L:basecosinus}, taking the trace we obtain
\[
\frac{1}{2}(\lambda_1^2+\mu_1^2+2\lambda_1\mu_1\frac{a_1}{b_1})  = \frac{1}{2}(\lambda_2^2+\mu_2^2+2\lambda_2\mu_2\frac{a_2}{b_2})\,.
\]
Hence
\begin{multline}\label{E:calculsommecarre3}
(\lambda_1^2+\mu_1^2) \cos\frac{2k_1\theta_1}{n_1}
 + 2(\lambda_1\mu_1+\mu_1^2\frac{a_1}{b_1}) \cos\frac{(n_1-2k_1)\theta_1}{n_1}\\
 =  (\lambda_2^2+\mu_2^2) \cos\frac{2k_2\theta_2}{n_2}
 + 2(\lambda_2\mu_2+\mu_2^2\frac{a_2}{b_2}) \cos\frac{(n_2-2k_2)\theta_2}{n_2}
\end{multline}
However, by Lemma~\ref{L:peuderacinescarre}, the number \eqref{E:calculsommecarre1} is irrational, hence the number in \eqref{E:calculsommecarre3} is not $0$; a contradiction with the minamility of $n_1$. Therefore $n_1=2$ and $k_1=1$, so $2k_1=n_1$, thus $2k_2=n_2$, the result follows for $x=\lambda_1+\mu_1$, and $y=\lambda_2+\mu_2$.
\end{cproof}

As $\cos\frac{\theta_1}{2} = \sqrt{\frac{b_1+a_1}{2b_1}}$ and $\cos\frac{\theta_2}{2} = \sqrt{\frac{b_2+a_2}{2b_2}}$, it follows from Claim~2 that $(b_1+a_1)(b_2+a_2)b_1b_2$ is a square; a contradiction. Therefore $K_1$ and $K_2$ are distincts.
\end{proof}

The following lemma is an immediate consequence of Hensel's Lemma (see for instance \cite[Theorem 5.6]{Narkiewicz}). 
    
\begin{lemma}\label{L:carremodpuissance}
Let $p\ge 3$ be a prime, let $q$ be relatively prime to $p$. Let $n\ge 1$ be an integer. Assume that $q$ is a square modulo $p$. Then $q$ is a square modulo $p^n$.
\end{lemma}

\begin{lemma}\label{L:carremodulopremier}
Let $t$ be an odd number. Then there are infinitely many prime number $p$ such that $p\equiv 2\mod t$, $p\equiv 3\mod 4$, and $2t$ is a square modulo $p$.
\end{lemma}

\begin{proof}
Given $p,q$ we denote by $\Legendre{p}{q}$ the Legendre symbol.
Set
\[
\varepsilon = \prod_{\substack{q|t\\ q\text{ prime}}} (-1)^{\frac{q-1}{2}}\Legendre{2}{q}\,.
\]
As $t$ is odd, it follows that $\varepsilon=\pm 1$. If $\varepsilon=1$ pick $p$ a prime number such that $p\equiv 2\mod t$ and $p\equiv 7\mod 8$. If $\varepsilon=-1$ pick $p$ a prime number such that $p\equiv 2\mod t$ and $p\equiv 3\mod 8$. In both case there are infinitely many primes satisfying those conditions by Dirichlet's theorem (cf. \cite{Dirichlet}), and in both case we have $p\equiv 3\mod 4$. Moreover by Gauss's quadratic reciprocity we have $\Legendre{2}{p}=\varepsilon$.

The following equalities hold
\begin{align*}
\Legendre{2t}{p} &= \Legendre{2}{p} \prod_{\substack{q|t\\ q\text{ prime}}} \Legendre{q}{p}\,, &&\text{by multiplicativity.}\\
&= \Legendre{2}{p} \prod_{\substack{q|t\\ q\text{ prime}}} (-1)^{\frac{p-1}{2}\frac{q-1}{2}}\Legendre{p}{q}\,, &&\text{by Gauss's quadratic reciprocity.}\\
&= \Legendre{2}{p} \prod_{\substack{q|t\\ q\text{ prime}}} (-1)^{\frac{q-1}{2}}\Legendre{2}{q}\,, &&\text{as $\frac{p-1}{2}$ is odd and $p\equiv 2\mod t$.}\\
&=\varepsilon^2\\
&=1
\end{align*}
Therefore $2t$ is a square modulo $p$.
\end{proof}

\begin{lemma}\label{L:thesquares}
Let $t \geq 1$ be a square-free odd number. There are infinitely many couples of positive integers $( m_i , n_i )$ such that, for all $i$, $m_i - 2$ and $n_i$ are both relatively prime to $2t$ and, for all $i \neq j$, $( 2^{m_i} t^{n_i} - 1 ) ( 2^{m_j} t^{n_j} - 1 )$ is not a square.
\end{lemma}

\begin{proof} 
Assume that we have $( m_1 , n_1 ),\dots,( m_k, n_k )$  (for $1\le i\le k$), such that, for every $i \neq j$, $( 2^{m_i} ( 2 t )^{n_i} - 1 ) ( 2^{m_j} ( 2t )^{n_j} - 1 )$ is not a square. 

It follows from Lemma~\ref{L:carremodulopremier} that there is a prime $p$ such that $p\equiv 2\mod t$, $p\equiv 3\mod 4$, $2t$ is a square modulo $p$ and, for all $1 \leq i \leq k$, $p$ does not divide  $2^{m_i} t^{n_i} - 1$. In particular $p$ is relatively prime to $2t$.

Pick $u\ge 1$ odd such that $(2t)^{\frac{p-1}{2}}<p^{u+1}$, hence $(2t)^{\frac{p-1}{2}}\not\equiv 1\mod p^{u+1}$. Note that Lemma~\ref{L:carremodpuissance} implies that $2t$ is a square modulo $p^{u+1}$. Let $x$ be such that $2t\equiv x^2\mod p^{u+1}$, and so $x$ is prime to $p$. Therefore $x^{p-1}\equiv (x^2)^{\frac{p-1}{2}}\equiv (2t)^{\frac{p-1}{2}}\mod p^{u+1}$, so $x^{p-1}\not\equiv 1\mod p^{u+1}$. However the multiplicative group $(\Z/p^{u+1}\Z)^*$ is cyclic of order $(p-1)p^u$, hence $x^{(p-1)p^u}\equiv 1\mod p^{u+1}$.

Note that $(p^u+1)^p\equiv 1\mod p^{u+1}$. As the subgroup of $(\Z/p^{u+1}\Z)^*$ generated by $x^{p-1}$ is not trivial and of order dividing $p^u$, it follows that this subgroup contains $p^u+1$. Therefore there is $c\ge 1$ such that $x^{(p-1)c}\equiv p^u+1\mod p^{u+1}$. As $x^{(p-1)p^u}\equiv 1\mod p^{u+1}$ we have $x^{(p-1)(c+p^u)}\equiv x^{(p-1)c}\equiv p^u+1\mod p^{u+1}$. Therefore as $p$ is odd, we can assume that $c$ is odd.

Set $v=\frac{p-1}{2}c$. Note that $2^v t^v=(2t)^{\frac{p-1}{2}c}\equiv x^{(p-1)c}\equiv p^u+1 \mod p^{u+1}$. As $(p-1)p^u$ is even and relatively prime to $t$ we can pick $w_0$ and $w_1$ odds such that $w_0-2$ is prime to $t$, $w_1$ is prime to $t$, $w_0\equiv v\mod (p-1)p^u$, and $w_1\equiv v\mod (p-1)p^u$. Therefore $2^{w_0}t^{w_1}\equiv 2^v t^v\equiv p^u+1\mod p^{u+1}$, so $p^u$ divides $2^{w_0}t^{w_1}-1$ and $p^{u+1}$ does not divides $2^{w_0}t^{w_1}-1$. 

Note that $v_p(2^{w_0}t^{w_1}-1)=u$ is odd, moreover $p$ does not appear in any prime decomposition of the $2^{m_i} t^{n_i} - 1$. Therefore for all $1\le i\le k$, $( 2^{m_i} t^{n_i} - 1 ) ( 2^{w_0} t^{w_1} - 1 )$ is not a square.
\end{proof}

\begin{theorem}\label{T:exemplenombreJR}
Let $t\ge 1$ be a square-free odd number. Then the following statement holds.
\begin{enumerate}
\item There are infinitely many fields $ K $ such that $\cO_K$ has Julia Robinson's number $\Ceil{2\sqrt{2t}}+2\sqrt{2t}$.
\item There are infinitely many fields $K$ such that $\cO_K$ has Julia Robinson's number $8t$.
\end{enumerate}
\end{theorem}

\begin{proof}
By Lemma~\ref{L:corpsdifferents} and Lemma~\ref{L:thesquares}, we can find infinitely many good $( a, b, r )$, such that $b= 2^{m-1} t^n$, with $m -2$ and $n$ relatively prime to $2t$, $a = b - 1$ and $r = 2t$, and all the totally real associated fields are distinct.

Fix such $(a,b,r)$. Let $K$ be the totally real field associated to $(a,b,r)$, let $s$ as in Notation~\ref{N:notJR2}. Therefore by Corollary~\ref{C:preencadrementJR}(4) we have $s^2=8t$, thus $s=2\sqrt{2t}$. Therefore by Theorem~\ref{T:MainJuliaRobinsonNumber} the Julia Robinson's number of $\cO_K$ is $\Ceil{2\sqrt{2t}}+2\sqrt{2t}$. So $(1)$ holds.

Set $r=k=2t$. Let $C$ as in Lemma~\ref{L:squarfreenumbers}, for $\varepsilon=\frac{1}{2}$. Consider $n\ge 1$ an odd integer such that $n$ is relatively prime to $t$ and $2^{n+1}t\ge \max(C,4t)$. Set $b=2^{n+1}t$. Note that there are infinitely many possibilities for such $b$. By Lemma~\ref{L:squarfreenumbers} there is $1\le a\le \frac{1}{2}b$ such that $a\equiv 1\mod k$, and $b-a$, $b+a$ are square-free. As $k|b$, it follows that $a$ is relatively prime to $b$.

Consider $K$ the field constructed in Section~\ref{S:JR}. Let $s$ as in Notation~\ref{N:notJR2}. As $b-a$ and $b+a$ are square-free, and $b\ge \frac{b}{2} + \frac{b}{2}\ge 2t+a$ by Corollary~\ref{C:preencadrementJR}(3) we have $s=4t$. Therefore by Theorem~\ref{T:MainJuliaRobinsonNumber} the Julia Robinson's number of $\cO_K$ is $\Ceil{4t}+4t=8t$. Again, by Lemma~\ref{L:corpsdifferents}, by choosing distinct $b$, we obtain distinct fields. So $(2)$ holds.
\end{proof}

\section{Acknowledgement}

The authors thank Francesco Amoroso for the invitation to the Universit\'e de Caen. The friendly ambiance and a discussion, with Francesco Amoroso, yields to the explicit example in Theorem~\ref{T:exemplenombreJR}(2). The authors also thank Florence and Jean Gillibert and Philippe Satg\'e for helpful discussions. Finally, the authors thank Denis Simon, for his very useful calculations using Paris, allowing us to find our examples.

\section{Funding}
The first author was supported by CONICYT, Proyectos Regulares FONDECYT no. 1150595, and by project number P27600 of the Austrian Science Fund (FWF). The second author was supported by CONICYT, Proyectos Regulares FONDECYT no. 1140946.

\end{document}